\documentclass[12pt]{amsart}

\newcommand{\nc}{\newcommand}

\usepackage[enableskew]{youngtab}
\def\KeyWord#1{$\backslash$\IfColor{$\!\!$\textRed{#1}\textBlack}{#1}$\!\!$}

\usepackage{amsfonts}
\usepackage{amssymb}
\usepackage{amsmath}
\usepackage{amsthm}

\usepackage{etex}
\reserveinserts{28}
\usepackage{amsgen}
\usepackage{amsopn}
\usepackage{verbatim}
\usepackage{xspace}
\usepackage{multicol}
\usepackage{eepic}
\usepackage{url}
\usepackage{upref}
\usepackage{pgf}
\usepackage{tikz}
\usetikzlibrary{patterns}
\usepackage{pgffor}
\usepackage{pgfcalendar}
\usepackage{pgfpages}
\usepackage{shuffle,yfonts}
\DeclareFontFamily{U}{shuffle}{}
\DeclareFontShape{U}{shuffle}{m}{n}{ <-8>shuffle7 <8->shuffle10}{}

\theoremstyle{plain}
\newtheorem{thm}{Theorem}[section]
\newtheorem{lem}[thm]{Lemma}
\newtheorem{prop}[thm]{Proposition}
\newtheorem{lem-defn}[thm]{Lemma-Definition}
\newtheorem{cor}[thm]{Corollary}
\newtheorem{conj}[thm]{Conjecture}

\theoremstyle{definition}

\newtheorem{defn}[thm]{Definition}
\newtheorem{rem}[thm]{Remark}

\newtheorem{eg}[thm]{Example}

\DeclareMathOperator{\dep}{{dp}}

\DeclareMathOperator{\sgn}{{sgn}}

\DeclareMathOperator{\ord}{{ord}}

\nc{\setA}{{\mathsf A}}
\nc{\setB}{{\mathsf B}}
\nc{\setC}{{\mathsf C}}
\nc{\setD}{{\mathsf D}}
\nc{\setE}{{\mathsf E}}
\nc{\setF}{{\mathsf F}}
\nc{\setG}{{\mathsf G}}
\nc{\setH}{{\mathsf{H}}}
\nc{\setI}{{\mathsf I}}
\nc{\setJ}{{\mathsf J}}
\nc{\setK}{{\mathsf K}}
\nc{\setL}{{\mathsf L}}
\nc{\setM}{{\mathsf M}}
\nc{\setN}{{\mathsf N}}
\nc{\setO}{{\mathsf O}}
\nc{\setP}{{\mathsf P}}
\nc{\setQ}{{\mathsf Q}}
\nc{\setR}{{\mathsf R}}
\nc{\setS}{{\mathsf S}}
\nc{\setT}{{\mathsf T}}
\nc{\setU}{{\mathsf U}}
\nc{\setV}{{\mathsf V}}
\nc{\setW}{{\mathsf W}}
\nc{\setX}{{\mathsf X}}
\nc{\setY}{{\mathsf Y}}
\nc{\setZ}{{\mathsf Z}}
\renewcommand{\a}{{\texttt{a}}}
\renewcommand{\b}{{\texttt{b}}}

\nc{\z}{{\texttt{z}}}

\nc{\exref}[1]{\chapref\ref{#1}}

\nc{\per}[1]{\underset{#1}{\boldsymbol \pi}\,}
\nc{\QSym}{{\mathsf{QSym}}}
\nc{\tla}{\overset{\leftarrow}}
\nc{\MT}{{\rm MT}}
\nc{\XX}{{X}}
\nc{\gF}{{\varPhi}}
\nc{\gL}{{\Lambda}}
\nc{\ot}{\otimes}

\nc{\dual}{\ast}
\nc{\wht}{\widehat}
\nc{\bwg}{{\bigwedge}}
\nc{\wg}{{\wedge}}
\nc{\mmu}{{\boldsymbol{\mu}}}
\nc{\mal}{{{\scstl \maltese}}}
\nc{\fA}{{\mathfrak A}}
\nc{\hfA}{{\widehat{\mathfrak A}}}
\nc{\HH}{{\mathbb H}}
\nc{\ra}{\rightarrow}
\nc{\os}{{\overset}}
\nc{\GG}{{\mathbb G}}
\nc{\F}{{\mathbb F}}
\nc{\Z}{{\mathbb Z}}
\nc{\R}{{\mathbb R}}
\nc{\N}{{\mathbb N}}
\nc{\ZN}{{{\mathbb N}_0}}
\nc{\Q}{{\mathbb Q}}
\nc{\CC}{{\mathbb C}}
\nc{\V}{{\mathbb V}}
\nc{\CP}{{\mathbb{CP}}}
\nc{\Cnn}{{\mathbb C}_{\ge 0}}
\nc{\Cp}{{\mathbb C}_{>0}}

\nc{\MHSS}{MH${}^\star$S\xspace}
\nc{\MHSSs}{MH${}^\star$Ss\xspace}
\nc{\MZSV}{MZ${}^\star$V\xspace}
\nc{\MZSVs}{MZ${}^\star$Vs\xspace}
\nc{\FMZSV}{FMZ${}^\star$V\xspace}
\nc{\FMZSVs}{FMZ${}^\star$Vs\xspace}
\nc{\FESS}{FE${}^\star$S\xspace}
\nc{\FESSs}{FE${}^\star$Ss\xspace}

\nc{\DSh}{{\mathsf{DSh}}}
\nc{\ShC}{{\mathsf{ShC}}}
\nc{\MZV}{{\mathsf{MZ}}}
\nc{\qMZ}{{q\mathsf{MZ}}}
\nc{\FMZ}{{\mathsf{FMZ}}}
\nc{\CMZV}{{\mathsf{CMZV}}}
\nc{\ES}{{\mathsf{ES}}}
\nc{\FES}{{\mathsf{FES}}}
\nc{\qMZV}{\mathsf{qMZV}}
\nc{\grqMZV}{\mathsf{grqMZV}}

\nc{\gemn}{{\mathfrak n}}

\nc{\wtcalM}{{\widetilde\calM}}
\nc{\gam}{{\gamma}}
\nc{\gG}{{\Gamma}}
\nc{\om}{{\omega}}
\nc{\vep}{{\varepsilon}}
\nc{\ga}{{\alpha}}
\nc{\gl}{{\lambda}}
\nc{\gb}{{\beta}}
\nc{\gd}{{\delta}}
\nc{\gf}{{\varphi}}
\nc{\orgd}{{\vec \gd\,}}
\nc{\gs}{{\sigma}}
\nc{\gth}{{\theta}}
\nc{\gS}{{\Sigma}}

\nc{\gk}{{\kappa}}

\nc{\tgz}{{\tilde{\zeta}}}
\nc{\gO}{{\Omega}}
\nc{\sif}{{\mathcal S}}
\nc{\gt}{{\tau}}
\nc{\Lra}{\Longrightarrow}
\nc{\lra}{\longrightarrow}
\nc{\lmaps}{\longmapsto}
\nc{\fS}{{\mathsf S}}
\nc{\DD}{{\mathfrak D}}
\nc{\Llra}{\Longleftrightarrow}
\nc{\ol}{\overline}
\def\ola#1{\overset{\text{\raisebox{-4pt}{$\scriptscriptstyle \leftarrow$}}}{#1}}

\nc{\zq}{{\zeta_q}}
\nc{\us}{\underset}
\nc{\tn}{{\tilde{n}}}
\nc{\gD}{{\Delta}}
\nc{\bi}{{\bf i}}
\nc{\bfone}{{\bf 1}}
\nc{\bfzero}{{\bf 0}}
\nc{\bftwo}{{\bf 2}}

\nc{\bfa}{{\bf a}}
\nc{\bfb}{{\bf b}}
\nc{\bfc}{{\bf c}}
\nc{\bfd}{{\bf d}}
\nc{\bfe}{{\bf e}}
\nc{\bferev}{\ola{\bf e}}
\nc{\bff}{{\bf f}}
\nc{\bfg}{{\bf g}}
\nc{\bfh}{{\bf h}}
\nc{\bfi}{{\bf i}}
\nc{\bfj}{{\bf j}}
\nc{\bfjrev}{\ola{\bf j}}
\nc{\obfi}{{\overrightarrow{\boldsymbol \imath}}}
\nc{\obfj}{{\overrightarrow{\boldsymbol \jmath}}}
\nc{\obfk}{{\overrightarrow{\bf k}}}
\nc{\veps}{{\varepsilon}}
\nc{\bfn}{{\bf n}}
\nc{\bfl}{{\bf l}}
\nc{\bfk}{{\bf k}}
\nc{\bfm}{{\bf m}}
\nc{\bfo}{{\bf o}}
\nc{\bfp}{{\bf p}}
\nc{\bfq}{{\bf q}}
\nc{\bfr}{{\bf r}}
\nc{\bfs}{{\bf s}}
\nc{\bft}{{\bf t}}
\nc{\bfu}{{\bf u}}
\nc{\bfv}{{\bf v}}
\nc{\bfw}{{\bf w}}
\nc{\bfx}{{\bf x}}
\nc{\bfy}{{\bf y}}
\nc{\bfz}{{\bf z}}
\nc{\bfA}{{\bf A}}
\nc{\bfB}{{\bf B}}
\nc{\bfC}{{\bf C}}
\nc{\bfD}{{\bf D}}
\nc{\bfE}{{\bf E}}
\nc{\bfF}{{\bf F}}
\nc{\bfG}{{\bf G}}
\nc{\bfH}{{\bf H}}
\nc{\bfI}{{\bf I}}
\nc{\bfJ}{{\bf J}}
\nc{\bfK}{{\bf K}}
\nc{\bfL}{{\bf L}}
\nc{\bfM}{{\bf M}}
\nc{\bfN}{{\bf N}}
\nc{\bfO}{{\bf O}}
\nc{\bfP}{{\bf P}}
\nc{\bfQ}{{\bf Q}}
\nc{\bfR}{{\bf R}}
\nc{\bfS}{{\bf S}}
\nc{\bfT}{{\bf T}}
\nc{\bfU}{{\bf U}}
\nc{\bfV}{{\bf V}}
\nc{\bfW}{{\bf W}}
\nc{\bfX}{{\bf X}}
\nc{\bfY}{{\bf Y}}
\nc{\bfZ}{{\bf Z}}
\nc{\bbA}{{\mathbb A}}
\nc{\bbB}{{\mathbb B}}
\nc{\bbC}{{\mathbb C}}
\nc{\bbD}{{\mathbb D}}
\nc{\bbE}{{\mathbb E}}
\nc{\bbF}{{\mathbb F}}
\nc{\bbG}{{\mathbb G}}
\nc{\bbH}{{\mathbb H}}
\nc{\bbI}{{\mathbb I}}
\nc{\bbJ}{{\mathbb J}}
\nc{\bbK}{{\mathbb K}}
\nc{\bbL}{{\mathbb L}}
\nc{\bbM}{{\mathbb M}}
\nc{\bbN}{{\mathbb N}}
\nc{\bbO}{{\mathbb O}}
\nc{\bbP}{{\mathbb P}}
\nc{\bbQ}{{\mathbb Q}}
\nc{\bbR}{{\mathbb R}}
\nc{\bbS}{{\mathbb S}}
\nc{\bbT}{{\mathbb T}}
\nc{\bbU}{{\mathbb U}}
\nc{\bbV}{{\mathbb V}}
\nc{\bbW}{{\mathbb W}}
\nc{\bbX}{{\mathbb X}}
\nc{\bbY}{{\mathbb Y}}
\nc{\bbZ}{{\mathbb Z}}
\nc{\bba}{{\mathbb a}}
\nc{\bbb}{{\mathbb b}}
\nc{\bbc}{{\mathbb c}}
\nc{\bbd}{{\mathbb d}}
\nc{\bbe}{{\mathbb e}}
\nc{\bbf}{{\mathbb f}}
\nc{\bbg}{{\mathbb g}}
\nc{\bbh}{{\mathbb h}}
\nc{\bbi}{{\mathbb i}}
\nc{\bbk}{{\mathbb K}}
\nc{\bbl}{{\mathbb l}}
\nc{\bbm}{{\mathbb m}}
\nc{\bbn}{{\mathbb n}}
\nc{\bbo}{{\mathbb o}}
\nc{\bbp}{{\mathbb p}}
\nc{\bbq}{{\mathbb q}}
\nc{\bbr}{{\mathbb r}}
\nc{\bbs}{{\mathbb s}}
\nc{\bbt}{{\mathbb t}}
\nc{\bbu}{{\mathbb u}}
\nc{\bbv}{{\mathbb v}}
\nc{\bbw}{{\mathbb w}}
\nc{\bbx}{{\mathbb x}}
\nc{\bby}{{\mathbb y}}
\nc{\bbz}{{\mathbb z}}

\nc{\calA}{{\mathcal A}}
\nc{\calB}{{\mathcal B}}
\nc{\calC}{{\mathcal C}}
\nc{\calD}{{\mathcal D}}
\nc{\calE}{{\mathcal E}}
\nc{\calF}{{\mathcal F}}
\nc{\calG}{{\mathcal G}}
\nc{\calH}{{\mathcal H}}
\nc{\calI}{{\mathcal I}}
\nc{\calJ}{{\mathcal J}}
\nc{\calK}{{\mathcal K}}
\nc{\calL}{{\mathcal L}}
\nc{\calM}{{\mathcal M}}
\nc{\calN}{{\mathcal N}}
\nc{\calO}{{\mathcal O}}
\nc{\calP}{{\mathcal P}}
\nc{\calQ}{{\mathcal Q}}
\nc{\calR}{{\mathcal R}}
\nc{\calS}{{\mathcal S}}
\nc{\calT}{{\mathcal T}}
\nc{\calU}{{\mathcal U}}
\nc{\calV}{{\mathcal V}}
\nc{\calW}{{\mathcal W}}
\nc{\calX}{{\mathcal X}}
\nc{\calY}{{\mathcal Y}}
\nc{\calZ}{{\mathcal Z}}

\usepackage[cal=boondoxo]{mathalfa}

 \nc{\cala}{{\mathcal a}}
\nc{\calb}{{\mathcal b}}
\nc{\calc}{{\mathcal c}}
\nc{\cald}{{\mathcal d}}
\nc{\cale}{{\mathcal e}}
\nc{\calf}{{\mathcal f}}
\nc{\calg}{{\mathcal g}}
\nc{\calh}{{\mathcal h}}
\nc{\cali}{{\mathcal i}}
\nc{\calj}{{\mathcal j}}
\nc{\calk}{{\mathcal k}}
\nc{\call}{{\mathcal l}}
\nc{\calm}{{\mathcal m}}
\nc{\caln}{{\mathcal n}}
\nc{\calo}{{\mathcal o}}
\nc{\calp}{{\mathcal p}}
\nc{\calq}{{\mathcal q}}
\nc{\calr}{{\mathcal r}}
\nc{\cals}{{\mathcal s}}
\nc{\calt}{{\mathcal t}}
\nc{\calu}{{\mathcal u}}
\nc{\calv}{{\mathcal v}}
\nc{\calw}{{\mathcal w}}
\nc{\calx}{{\mathcal x}}
\nc{\caly}{{\mathcal y}}
\nc{\calz}{{\mathcal z}}

\nc{\frakA}{{\mathfrak A}}
\nc{\frakB}{{\mathfrak B}}
\nc{\frakC}{{\mathfrak C}}
\nc{\frakD}{{\mathfrak D}}
\nc{\frakE}{{\mathfrak E}}
\nc{\frakF}{{\mathfrak F}}
\nc{\frakG}{{\mathfrak G}}
\nc{\frakH}{{\mathfrak H}}
\nc{\frakI}{{\mathfrak I}}
\nc{\frakJ}{{\mathfrak J}}
\nc{\frakK}{{\mathfrak K}}
\nc{\frakL}{{\mathfrak L}}
\nc{\frakM}{{\mathfrak M}}
\nc{\frakN}{{\mathfrak N}}
\nc{\frakO}{{\mathfrak O}}
\nc{\frakP}{{\mathfrak P}}
\nc{\frakQ}{{\mathfrak Q}}
\nc{\frakR}{{\mathfrak R}}
\nc{\frakS}{{\mathfrak S}}
\nc{\frakT}{{\mathfrak T}}
\nc{\frakU}{{\mathfrak U}}
\nc{\frakV}{{\mathfrak V}}
\nc{\frakW}{{\mathfrak W}}
\nc{\frakX}{{\mathfrak X}}
\nc{\frakY}{{\mathfrak Y}}
\nc{\frakZ}{{\mathfrak Z}}
\nc{\fraka}{{\mathfrak a}}
\nc{\frakb}{{\mathfrak b}}
\nc{\frakc}{{\mathfrak c}}
\nc{\frakd}{{\mathfrak d}}
\nc{\frake}{{\mathfrak e}}
\nc{\frakf}{{\mathfrak f}}
\nc{\frakg}{{\mathfrak g}}
\nc{\frakh}{{\mathfrak h}}
\nc{\fraki}{{\mathfrak i}}
\nc{\frakj}{{\mathfrak j}}
\nc{\frakk}{{\mathfrak k}}
\nc{\frakl}{{\mathfrak l}}
\nc{\frakm}{{\mathfrak m}}
\nc{\frakn}{{\mathfrak n}}
\nc{\frako}{{\mathfrak o}}
\nc{\frakp}{{\mathfrak p}}
\nc{\frakq}{{\mathfrak q}}
\nc{\frakr}{{\mathfrak r}}
\nc{\fraks}{{\mathfrak s}}
\nc{\frakt}{{\mathfrak t}}
\nc{\fraku}{{\mathfrak u}}
\nc{\frakv}{{\mathfrak v}}
\nc{\frakw}{{\mathfrak w}}
\nc{\frakx}{{\mathfrak x}}
\nc{\fraky}{{\mathfrak y}}
\nc{\frakz}{{\mathfrak z}}
\nc{\so}{{\mathfrak so}}
\nc{\slfour}{{\mathfrak sl}_4}
\nc{\one}{{\bf 1}}
\nc{\zero}{{\bf 0}}

\nc{\sha}{\shuffle}
\nc{\inj}{\hookrightarrow}

\nc{\zetas}{{\zeta^\star}}

\nc{\DLN}{{\mathsf{DivLog}_\N}}
\nc{\DLD}{{\mathsf{DivLog}_D}}

\nc{\invdots}{{.\text{\raisebox{0.2ex}{$\cdot$}} \text{\raisebox{0.9ex}{$\cdot$}} }}
\nc{\Sy}{{\mathcal S}}
\nc{\inv}{{\rm inv}}
\nc{\Rac}{{\mathcal R}}
\nc{\dd}{{\mathfrak d}}
\nc{\shD}{{\mbox{\cyr D}}}
\nc{\tz}{\tilde{\zeta}}

\nc{\x}{{\mathtt{x}}}
\nc{\y}{{\mathtt{y}}}
\nc{\Qxy}{\Q\langle \x,\y\rangle}

\DeclareMathOperator{\singL}{{L}}

\nc{\recurO}{{{\mathfrak D}_{z,\barz}}}
\nc{\recurI}{{{\mathfrak I}_{z,\barz}}}
\nc{\app}{{\sharp}}
\nc{\barz}{{\bar{z}}}
\nc{\bB}{{\mathsf{B}}}
\nc{\angX}{{\langle\!\langle \setX\rangle\!\rangle}}
\nc{\revs}{\ola}
\nc{\res}{{\rm{res}}}
\nc{\sv}{{\rm{sv}}}
\nc{\tdx}{{\y}}
\nc{\setLs}{{\mathsf{Ls}}}

\nc{\eps}{{\epsilon}}

\nc{\lsetH}{{\hat{\setH}}}
\nc{\Sooo}{\mathsf{S}^0_{0,0}\,}
\nc{\Solo}{\mathsf{S}^0_{1,0}\,}
\nc{\Sool}{\mathsf{S}^0_{0,1}\,}
\nc{\Soll}{\mathsf{S}^0_{1,1}\,}
\nc{\Sllo}{\mathsf{S}^1_{1,0}\,}
\nc{\Slol}{\mathsf{S}^1_{0,1}\,}
\nc{\Sloo}{\mathsf{S}^1_{0,0}\,}
\nc{\Slll}{\mathsf{S}^1_{1,1}\,}
\nc{\Stot}{\mathsf{S}_\setX}
\nc{\Stotnx}{\mathsf{S}}
\nc{\Zooo}{\Phi^0_{0,0}\,}
\nc{\Zolo}{\Phi^0_{1,0}\,}
\nc{\Zool}{\Phi^0_{0,1}\,}
\nc{\Zloo}{\Phi^1_{0,0}\,}
\nc{\Zllo}{\Phi^1_{1,0}\,}
\nc{\Zlol}{\Phi^1_{0,1}\,}
\nc{\Zoll}{\Phi^0_{1,1}\,}
\nc{\vooo}{V^0_{0,0}\,}
\nc{\volo}{V^0_{1,0}\,}
\nc{\vool}{V^0_{0,1}\,}
\nc{\vloo}{V^1_{0,0}\,}
\nc{\vllo}{V^1_{1,0}\,}
\nc{\vlol}{V^1_{0,1}\,}
\nc{\voll}{V^0_{1,1}\,}
\nc{\vlll}{V^1_{1,1}\,}
\nc{\Fo}{F}

\nc{\lSooo}{\hat{\mathsf{S}}^0_{0,0}\,}
\nc{\lSolo}{\hat{\mathsf{S}}^0_{1,0}\,}
\nc{\lSool}{\hat{\mathsf{S}}^0_{0,1}\,}
\nc{\lSoll}{\hat{\mathsf{S}}^0_{1,1}\,}
\nc{\lSllo}{\hat{\mathsf{S}}^1_{1,0}\,}
\nc{\lSlol}{\hat{\mathsf{S}}^1_{0,1}\,}
\nc{\lSloo}{\hat{\mathsf{S}}^1_{0,0}\,}
\nc{\lSlll}{\hat{\mathsf{S}}^1_{1,1}\,}
\nc{\lStot}{{\hat{\mathsf{S}}_\setX}}
\nc{\lStotnx}{\hat{\mathsf{S}}}
\nc{\lZooo}{\hat{\Phi}^0_{0,0}\,}
\nc{\lZolo}{\hat{\Phi}^0_{1,0}\,}
\nc{\lZool}{\hat{\Phi}^0_{0,1}\,}
\nc{\lZloo}{\hat{\Phi}^1_{0,0}\,}
\nc{\lZllo}{\hat{\Phi}^1_{1,0}\,}
\nc{\lZlol}{\hat{\Phi}^1_{0,1}\,}
\nc{\lvooo}{\hat{V}^0_{0,0}\,}
\nc{\lvolo}{\hat{V}^0_{1,0}\,}
\nc{\lvool}{\hat{V}^0_{0,1}\,}
\nc{\lvloo}{\hat{V}^1_{0,0}\,}
\nc{\lvllo}{\hat{V}^1_{1,0}\,}
\nc{\lvlol}{\hat{V}^1_{0,1}\,}
\nc{\lvoll}{\hat{V}^0_{1,1}\,}
\nc{\lvlll}{\hat{V}^1_{1,1}\,}
\nc{\lZo}{\hat{\Phi}_0}
\nc{\lZp}{\hat{\Phi}_{\pi}}
\nc{\lZH}{\hat{\Phi}_{H}}
\nc{\lZs}{\hat{\Phi}_{s}}
\nc{\lF}{\hat{F}}
\nc{\lH}{{\hat{H}}}

\nc{\gFF}{{\chi}}
\nc{\gX}{{\varPsi}}
\nc{\gXs}{\gX^\star}
\nc{\cv}{{\rm cv}}
\nc{\myone}{{1}}
\nc{\whH}{\widehat{H}}
\nc{\whR}{{R}}

\nc{\KZ}{{\rm KZ}}
\nc{\even}{{\rm ev}}
\nc{\odd}{{\rm od}}

\nc{\na}{\natural}
\nc{\tT}{\tilde{T}}

\nc{\db}{{\mathbb D}}

\nc{\BTF}{{F}}
\nc{\gbb}{{\beta}}

\nc{\bfzt}{{{\boldsymbol \zeta}}}
\nc{\tf}{{\tilde{f}}}

\nc{\sHq}{{\mathfrak h}}
\nc{\tHq}{{\tilde{H}}}
\nc{\stHq}{{\tilde{\sHq}}}
\nc{\hHq}{{\hat{H}}}
\nc{\shHq}{{\hat{\sHq}}}

\nc{\bigcc}{\operatorname{\phantom{a} \text{\raisebox{1pt}{$\scstl \circ$}}\hskip-2.2ex\bigsqcup}}

\nc{\hone}{{\widehat{1}}}
\nc{\genf}{\genfrac{[}{]}{0pt}{}}

\nc{\oll}[1]{\underline{#1}}
\nc{\hyf}{{\text{-}}}
\nc{\bt}{{\bf 2}}
\nc{\vbfp}{{\underline{\bfp}}}
\nc{\vbfs}{{\underline{\bfs}}}
\nc{\vbfm}{{\underline{\bfm}}}
\nc{\vbfw}{{\underline{\bfw}}}

\nc{\wbfp}{{\widetilde{\bfp}}}
\nc{\wbfs}{{\widetilde{\bfs}}}
\nc{\wbfw}{{\widetilde{\bfw}}}
\nc{\wbfm}{{\widetilde{\bfm}}}
\nc{\wdt}{{\widetilde{t}}}
\nc{\wdl}{{\widetilde{\gl}}}
\nc{\wdp}{{\widetilde{p}}}
\nc{\wwbfw}{{\overset{\text{\raisebox{-2pt}{$\approx$}}}{\bfw}}}
\nc{\wwbfp}{{\overset{\text{\raisebox{-2pt}{$\approx$}}}{\bfp}}}
\nc{\wwbfs}{{\overset{\text{\raisebox{-2pt}{$\approx$}}}{\bfs}}}
\nc{\wwbfm}{{\overset{\text{\raisebox{-2pt}{$\approx$}}}{\bfm}}}
\nc{\wwdp}{{\overset{\text{\raisebox{-2pt}{$\approx$}}}{\!p}}}
\nc{\wwdl}{{\overset{\text{\raisebox{-2pt}{$\approx$}}}{\!\gl}}}
\nc{\tb}{{\tilde{b}}}
\nc{\tB}{{\widetilde{B}}}
\nc{\tX}{{\widetilde{X}}}
\nc{\tY}{{\widetilde{Y}}}
\nc{\tbfs}{{\tilde{\bfs}}}
\nc{\tbft}{{\tilde{\bft}}}
\nc{\tbfu}{{\tilde{\bfu}}}
\nc{\ttbfs}{{\hat{\bfs}}}
\nc{\ttbft}{{\hat{\bft}}}
\nc{\ttbfu}{{\hat{\bfu}}}
\nc{\rrho}{{\hat{\rho}}}
\nc{\ggk}{{\hat{\gk}}}
\nc{\ggs}{{\hat{\gs}}}
\nc{\oI}{{\overline{I}}}
\nc{\bI}{{\bar{I}}}
\nc{\bJ}{{\bar{J}}}
\nc{\bK}{{\bar{K}}}

\nc{\bfgb}{{\boldsymbol \gb}}
\nc{\bfgl}{{\boldsymbol \gl}}
\nc{\wbfgl}{{\widetilde{\bfgl}}}
\nc{\wwbfgl}{{\overset{\text{\raisebox{-2pt}{$\approx$}}}{\bfgl}}}
\nc{\bfga}{{\boldsymbol \ga}}
\nc{\bfgs}{{\boldsymbol \gs}}
\nc{\bfxi}{{\boldsymbol \xi}}
\nc{\bfmu}{{\boldsymbol \mu}}
\nc{\bfnu}{{\boldsymbol \nu}}
\nc{\bftau}{{\boldsymbol \tau}}
\nc{\bfchi}{{\boldsymbol \chi}}

\nc{\Czeta}{{\calC}}
\nc{\SC}{{S}}
\nc{\htt}{{\rm ht}}
\nc{\uar}{{\uparrow}}
\nc{\dar}{{\downarrow}}

\nc{\RC}{{R}}
\nc{\QX}{{\Q\langle \setX\rangle}}
\nc{\QY}{{\Q\langle \setY\rangle}}
\nc{\CX}{{\CC\langle \setX\rangle}}
\nc{\CY}{{\CC\langle \setY\rangle}}
\nc{\RX}{{R\langle \setX\rangle}}
\nc{\RY}{{R\langle \setY\rangle}}
\nc{\QXX}{{\Q\langle\!\langle \setX\rangle\!\rangle}}
\nc{\QYY}{{\Q\langle\!\langle \setY\rangle\!\rangle}}
\nc{\CXX}{{\CC\langle\!\langle \setX\rangle\!\rangle}}
\nc{\CYY}{{\CC\langle\!\langle \setY\rangle\!\rangle}}
\nc{\RXX}{{R\langle\!\langle \setX\rangle\!\rangle}}
\nc{\RYY}{{R\langle\!\langle \setY\rangle\!\rangle}}
\nc{\lXr}{{\langle\!\langle \setX\rangle\!\rangle}}
\nc{\lYr}{{\langle\!\langle \setY\rangle\!\rangle}}
\nc{\Cat}[2]{\mathop{\mathrm{\bf Cat}}\limits_{#1}^{#2}}
\nc{\ones}{\{1\}}
\nc{\ud}{{\rm d}}

\title{Finite Multiple zeta Values and \\ Finite Euler Sums}
\author{Jianqiang Zhao}
\date{}

\begin{document}

\maketitle

\begin{abstract}
The alternating multiple harmonic sums are partial sums of the
iterated infinite series defining the Euler sums which are
the alternating version of the multiple zeta values. In this paper,
we present some systematic structural results of the van Hamme type
congruences of these sums, collected as finite Euler sums.
Moreover, we relate this to the
structure of the Euler sums which generalizes the corresponding
result of the multiple zeta values. We also provide a few
conjectures with extensive numerical support.
\end{abstract}



\section{Introduction and preliminaries}
A very fruitful practice in producing interesting congruences
is to consider so-called van Hamme type congruences.
One starts with a well-behaved infinite series whose summands are
given by rational numbers and then, for suitable primes $p$,
looks at the $(p-1)$-st partial sum modulo $p$, or even
modulo higher powers $p$ which leads to super congruences.

A classical result in this spirit is a variant of
Wolstenholme's Theorem \cite{Wolstenholme1862} dating back
to the mid nineteenth century: for all prime $p\ge 5$ we have
\begin{equation}\label{equ:wols}
\sum_{k=1}^{p-1} \frac1{k^2}\equiv 0 \pmod{p},\qquad
\sum_{k=1}^{p-1} \frac1k\equiv 0 \pmod{p^2}.
\end{equation}

Congruences in \eqref{equ:wols} were improved to super congruences
later (see \cite{Glaisher1900b,Sun2000}) in which
one finds that Bernoulli numbers play the key roles
by virtue of the following Faulhaber's formula: (see \cite[\S\,23.1.4-7]{AbramowitzSt1972})
\begin{equation}\label{equ:sump}
\sum_{j=1}^{n-1} j^d=\sum_{r=0}^d {d+1\choose r}\frac{B_r}{d+1}
n^{d+1-r}, \quad \forall n,d\ge 1.
\end{equation}

Generalizing the congruences in \eqref{equ:wols} to
multiple harmonic sums (MHSs) was
the initial motivation for our work in \cite{Zhao2003}.
To define these sums and their alternating version, we start by
looking at a sort of double cover of the set $\N$ of positive integers.
Let $\db$ be the set of \emph{signed numbers}
\begin{equation}\label{equ:dbDefn}
 \db:=\N \cup \ol{\N}, \quad \text{where}\quad \ol{\N}=\{\bar k: k\in\N\}.
\end{equation}
Define the absolute value function $| \cdot |$ on $\db$ by
$|k|=|\bar k|=k$ for all $k\in\N$ and the sign function by
$\sgn(k)=1$ and $\sgn(\bar k)=-1$ for all $k\in\N$.
On $\db$ we define a commutative and associative
binary operation $\oplus$ (called \emph{O-plus})
as follows: for all $a,b\in\db$
\begin{equation}\label{equ:oplusDefn}
    a\oplus b=
\left\{
  \begin{array}{ll}
    \ol{|a|+|b|}, \quad& \hbox{if $\sgn(a)\ne \sgn(b)$;} \\
    |a|+|b|, & \hbox{if $\sgn(a)=\sgn(b)$.} \\
  \end{array}
\right.
\end{equation}

For any $d\in \N$ and $\bfs=(s_1, \ldots, s_d)\in \db^d$ we define
the \emph{alternating multiple harmonic sums} (AMHSs) by
\begin{align*}
H_n(s_1, \ldots, s_d)=&\sum_{n\ge k_1>k_2>\ldots>k_d\ge 1}\prod_{j=1}^d
\frac{\sgn (s_j)^{k_j}}{k_j^{|s_j|}},\\
H^\star_n(s_1, \ldots, s_d)=&\sum_{n\ge k_1\ge k_2\ge\ldots\ge k_d\ge 1}\prod_{j=1}^d
\frac{\sgn (s_j)^{k_j}}{k_j^{|s_j|}}.
\end{align*}
When $\bfs\in \N^d$ both of these have been called
\emph{multiple harmonic sums} (MHSs) in the literature. But more precisely,
the second should be called \emph{multiple harmonic star sums}.
Conventionally, the number $d$ is called the \emph{depth}, denoted by $\dep(\bfs)$, and $|\bfs|:=\sum_{j=1}^d |s_j|$ the \emph{weight}.
For convenience we set  $H_n(\bfs)=0$ if $n<\dep(\bfs)$,
$H_n(\emptyset)=H^\star_n(\emptyset)=1$ for all $n\ge 0$.
To save space, we put $\{s\}^d=(s,\dots,s)$ with $s$ repeating $d$ times.

For example, we have
\begin{thm}\label{thm:homogeneousWols} \emph{(\cite[Theorem~2.13]{Zhao2008a})}
Let $s$ and $d$ be two positive integers. Let $p$ be an odd prime
such that $p\ge d+2$ and $p-1$ divides none of $ds$ and $ks+1$ for
$k=1,\dots, d$. Then
$$H_{p-1}(\{s\}^d) \equiv
\left\{
  \begin{array}{ll}
    0 \pmod{p},   &  \text{if the weight $ds$ is even;} \\
    0 \pmod{p^2}, \quad &  \text{if the weight $ds$ is odd.}
  \end{array}
\right.
$$
In particular, the above are always true if $p\ge ds+3$.
\end{thm}

Congruences in \eqref{equ:wols} have also been generalized
to some other non-homogeneous MHSs in \cite{Hoffman2004,Zhao2008a}
and further to the alternating version
in \cite{TaurasoZh2010,Zhao2011c}. In particular, in \cite{Zhao2011c}
we defined an adele-like structure in which MHSs modulo primes $p$
are collected and form some objects which, after Kaneko and Zagier,
we call, in this paper, the \emph{finite multiple zeta values} (FMZVs) and
the \emph{finite Euler sums} (FESs). See Definition~\ref{defn:FES} for
the precise meaning.

When $\bfs\in\db^d$ we obtain the Euler (star) sums by
taking the limit of AMHSs:
\begin{equation*}
\zeta(\bfs)=\lim_{n\to\infty} H_n(\bfs),\quad
\zeta^\star(\bfs)=\lim_{n\to\infty} H^\star_n(\bfs).
\end{equation*}
When $\bfs\in\N^d$ these become the \emph{multiple zeta values} (MZVs) and
\emph{multiple zeta star values} (MZSVs), respectively.

In recent years, MZVs and Euler sums have appeared in many areas
of mathematics and mathematical physics. The main theorem
obtained by Brown in \cite{Brown2012} implies that every period
of the mixed Tate motives unramified over $\Z$ is a
$\Q[\frac1{2\pi i}]$-linear combination of the MZVs.
It also implies that every MZV is a $\Q$-linear combination
of the \emph{Hoffman elements}, i.e., MZVs with arguments
equal to 2 or 3. Let $\MZV_n$ be the $\Q$-vector space spanned
by the MZVs of weight $n$. Then Brown's result
shows that the $\Q$-dimension $\dim \MZV_n\le d_n$ where the
Padovan numbers $d_n$ has the generating series
\begin{equation*}
    \frac{1}{1-t^2-t^3}=\sum_{n=0}^\infty d_n t^n.
\end{equation*}
Zagier conjectured that in fact $\dim_\Q \MZV_n=d_n$.

Similarly, let $\ES_n$ be the $\Q$-vector space spanned
by all the weight $n$ Euler sums. A conjecture similar to that
of Zagier made by Broadhurst \cite{Broadhurst1996}  says that
\begin{equation}\label{equ:BroadhurstConj}
\sum_{n=0}^\infty (\dim_\Q \ES_n) t^n=\frac{1}{1-t-t^2}.
\end{equation}
Hence $\dim_\Q \ES_n$ should be just Fibonacci numbers.
This has been proved by Deligne \cite{Deligne2010}
under the assumption of a variant Grothendieck's period conjecture.

Very recently, a surprising connection between the
FMZVs (see Definition~\ref{defn:FES}) and MZVs has emerged.
A similar connection between AMHSs and Euler sums should also exist.
We will present these in \S\,\ref{sec:dimConjMZV} and
\S\,\ref{sec:dimConjES} after recalling
many related results and proving some new ones in
\S\S\,\ref{sec:adeleStructure}-\ref{sec:duality}.

The primary goal of this paper is to continue our study of the
congruence properties of MHSs and AMHSs initiated in
\cite{Hoffman2004,TaurasoZh2010,Zhao2008a,Zhao2011c}.
The new results come from two sources: one is the algebra approach
using quasi-symmetric functions with signed powers
(see \S\,\ref{sec:QSym}),
the other is the shuffle relation first discovered by Kaneko for
MHSs and generalized here to AMHSs (see \S\,\ref{sec:FESshuffle}).

\noindent
{\bf Acknowledgement.} This paper was written while the author was
visiting the Max Planck Institute for Mathematics,
L'Institut des Hautes Etudes Scientifiques and
the National Taiwan University in 2015.
He would like to thank F.\ Brown, M.\ Kaneko, M.\ Kontsevich and D.\ Zagier
for a few very enlightening conversations. He also thanks Kh. Hessami Pilehrood
and T. Hessami Pilehrood for their helpful comments.
This work was partly supported by the USA NSF grant DMS-1162116.

\section{An adele-like structure}\label{sec:adeleStructure}
We defined an adele-like structure in \cite{Zhao2011c} as a
suitable theoretical framework in which the
van Hamme type congruences for MHSs and AMHSs can be organized
and further investigated. Let $\calP$ be the set of primes
and let $\ell$ be a positive integer. For any $n\in\N$, put
\begin{equation}\label{equ:poormanA}
\calA_\ell(n):=\prod_{p\in\calP, p\ge n}(\Z/p^\ell\Z),\quad
\calA_\ell:=\prod_{p\in\calP}(\Z/p^\ell\Z)\bigg/\bigoplus_{p\in\calP}(\Z/p^\ell\Z),
\end{equation}
with componentwise addition and multiplication.
We call $\ell$ the \emph{superbity}. This object seems to appear first
in \cite{Kontsevich2009}.

\begin{rem}
Note that there is a natural projection $\calA_\ell(n)\to \calA_\ell$
by inserting $0$'s in the components corresponding to $p<n$. So every
identity that holds in $\calA_\ell(n)$ for any particular $n$
is still valid in $\calA_\ell$,
but not vice versa. Hence, results obtained
in \cite{Zhao2011c} are more precise than the corresponding ones
stated in the framework of this paper.
\end{rem}

Suppose the elements of $\calA_\ell$ are represented by $(a_p):=(a_p)_{p\in\calP}$.
Observe that any two elements $(a_p)$ and $(b_p)$ represent the same one in
$\calA_\ell$ if and only if $a_p=b_p$ for all but finitely many primes $p$.
It is straightforward to see that $\calA_\ell$ is a $\Q$-algebra after
embedding $\Q$ in $\calA_\ell$ ``diagonally'' using the following map:
\begin{equation}\label{equ:iotaForQ}
\begin{split}
  \iota_\ell:  \Q &\,  \lra \quad   \calA_\ell \\
                r &\, \lra   \big(\iota_{p,\ell}(r) \big)_{p\in\calP}  \\
\end{split}
\end{equation}
where $ \iota_\ell(0)=(0)_{p\in\calP} $ and for all nonzero $r\in \Q$
\begin{equation*}
     \iota_{p,\ell} (r):=\left\{
                   \begin{array}{ll}
                     0, & \hbox{if $\ord_p(r)<0$;} \\
                     r \pmod{p^\ell}, \quad & \hbox{if $\ord_p(r)\ge 0$.}
                   \end{array}
                 \right.
\end{equation*}

\begin{prop}\label{prop:embed} \emph{(cf.\ \cite[Lemma 2.1]{Zhao2011c})}
The map $\iota_\ell$ is a monomorphism of algebras. Namely, $\iota_\ell$
gives an embedding of $\Q$
into $\calA_\ell$ as an algebra.
\end{prop}

\begin{proof}
Clear.
\end{proof}

\begin{rem}
It is easy to see that the cardinality of $\calA_\ell$ is
$\aleph_1$, \emph{i.e.}, the same as that of the  real numbers.
It is not too hard to construct an
embedding of $\R$ into $\calA_\ell$ as a set. But we don't
know whether there is an embedding of $\R$ as an algebra.
\end{rem}

The following conventions are quite convenient for us. First,
we abuse our notation by writing $a_p$ for
$(a_p)_{p\in\calP}$ if no confusion arises. In particular,
whenever $p$ appears in an equation in
$\calA_\ell$ ($\ell\ge 2$) it means the element $(p)_{p\in\calP}$.
Second, by writing $(a_p)_{p\ge k}\in \calA_\ell$ we mean
the element $(a_p)_{p\in\calP}$ with $a_p=0$ for all $p<k$.

\begin{defn}
We define a number $a\in \calA_\ell$ to be \emph{algebraic over} $\Q$
if there is a nontrivial polynomial $f(t)\in\Q[t]$ such that $f(a)=0$.
A non-algebraic number of $\calA_\ell$ over $\Q$ is called
\emph{transcendental}. Finitely many numbers
$a_1,\dots,a_n\in\calA_\ell$ are called
\emph{algebraically independent over} $\Q$ if for any nontrivial
polynomial $f(t_1,\dots,t_n)\in\Q[t_1,\dots,t_n]$ we have
$f(a_1,\dots,a_n)\ne 0.$
We call the elements in an infinite subset $S$ of $\calA_\ell$
\emph{algebraically independent over} $\Q$ if any finitely many
elements of $S$ are.
\end{defn}

\begin{defn}
For any non-negative integer $k$, we define
the $\calA_1$-\emph{Bernoulli numbers}
$$\gb_k:=\Big(\frac{B_{p-k}}k  \pmod{p}\Big)_{p>k}. $$
For all $k\ge 2$, we define
the $k$th $\calA_1$-Fermat quotient
\begin{equation*}
q_k:=\Big(\frac{k^{p-1}-1}{p}  \pmod{p} \Big)_{p>k}.
\end{equation*}
\end{defn}

We see that $\gb_{2k}=0$ for all $k\ge 1$ since all Bernoulli-number
$B_{2j+1}=0$ when $j\ge 1$. As for the $\calA_1$-Fermat quotient $q_2$,
according to \cite{KnauerRi2005}, we know that for all primes less than
$1.25\times 10^{15}$ the super congruence $2^{p-1}\equiv 1\pmod{p^2}$ is satisfied
by only two primes $1093$ and $3511$ which are called Wieferich primes.
It is also known that $q_2\ne 0$ in $\calA_1$ under $abc$-conjecture
(see \cite{Silverman1998}).
Moreover, note that we have $q_{k\ell}=q_k+q_\ell$
for all $k,\ell$ and $q_1=0$. So  $q_k$ is an $\calA_1$-analog
of the logarithm value $\log k$.

\begin{conj}\label{conj:A_1BernoulliAlgIndpt}
Put $\gb_1=1$ by abuse of notation.
Suppose $n_1,\dots,n_r\in\N$ such that
$\log n_1,\dots,\log n_r$ are $\Q$-linearly independent. Then
the numbers in the set
\begin{equation*}
\bigcup_{k=0}^\infty
\Big\{\gb_{2k+1}q_{n_1},\dots,\gb_{2k+1}q_{n_r}  \Big\}
\end{equation*}
are algebraically independent.
\end{conj}

One should compare this with the following
\begin{conj}\label{conj:BernoulliAlgIndpt}
Put $\zeta(1)=1$ by abuse of notation.
Suppose $n_1,\dots,n_r\in\N$ such that
$\log n_1,\dots,\log n_r$ are $\Q$-linearly independent. Then
the numbers in the set
\begin{equation*}
\bigcup_{k=0}^\infty
\Big\{\zeta(2k+1)\log n_1,\dots,\zeta(2k+1)\log n_r\Big\}
\end{equation*}
are algebraically independent.
\end{conj}

Like $\calA_1$-Bernoulli numbers and $\calA_1$-Fermat quotients,
every object similarly defined for each prime can be
put into $\calA_1$. Here is a very short list for such numbers
that we believe are of some interest.

\begin{itemize}
\item
Wilson quotients
$\frac{(p-1)!+1}{p}$
\item
$p$-adic $\Gamma$ values at any
rational point
$\Gamma_p(a/b)$
\item
$F_{p-(\frac{5}{p})}/p$: the Fibonacci quotient (OEIS A092330)
\item
$U_{p-(\frac{2}{p})}(2, -1)/p$: the Pell quotient (OEIS A000129)
\item
$U_{p-(\frac{3}{p})}(4, 1)/p$: a quotient related to the Lucas sequence 1, 4, 15, 56, 209, \dots (OEIS A001353)
\item
$U_{p-(\frac{6}{p})}(10, 1)/p$: a quotient related to the Lucas sequence 1, 10, 99, 980, 9701, \dots (OEIS A004189)
\end{itemize}

\begin{defn}\label{defn:FES}
For $\bfs=(s_1,\dots,s_d)\in\db^d$ (see \eqref{equ:dbDefn}), we define
\begin{align*}
\zeta_{\calA_\ell}(\bfs)
  &=\sum_{p>k_1>\dots>k_d\ge1}\frac{\sgn(s_1)^{k_1}\cdots \sgn(s_d)^{k_d}}
								{k_1^{|s_1|}\cdots k_d^{|s_d|}}\in \calA_\ell,\\
\zeta_{\calA_\ell}^{\star}(\bfs)
  &=\sum_{p>k_1\ge\dots\ge k_n\ge1}\frac{\sgn(s_1)^{k_1}\cdots \sgn(s_d)^{k_d}}
									{k_1^{|s_1|}\dotsm k_d^{|s_d|}}\in \calA_\ell.
\end{align*}
These elements in $\calA_\ell$ are called \emph{finite Euler sums}
(FESs) of superbity $\ell$. Further, if all the $s_j$'s are positive integers then,
after Kaneko and Zagier, we call them \emph{finite multiple zeta values}
(FMZVs) of superbity $\ell$.
\end{defn}

\begin{rem}
We have intentionally dropped the word ``star'' for the
$\zeta_{\calA_\ell}^{\star}$-values since we will see that
they can be expressed by all the $\zeta_{\calA_\ell}$-values.
See \eqref{equ:FMZVSH} in Theorem~\ref{thm:Hoff}.
\end{rem}

For example, Theorem~\ref{thm:homogeneousWols}  can be rephrased as
\begin{thm}\label{thm:homogeneousWols1}
Let $s$ and $d$ be two positive integers. Then
\begin{align}\label{equ:homogeneousWols}
\zeta_{\calA_1}(\{s\}^d)=0 \quad &\text{if the weight $ds$ is even;}\\
\zeta_{\calA_2}(\{s\}^d)=0 \quad &\text{if the weight $ds$ is odd.}\phantom{a}
\label{equ:homogeneousWolsOddWeight}
\end{align}
\end{thm}

Theorem~\ref{thm:homogeneousWols1} has been generalized to higher superbities
by Zhou and Cai.
\begin{thm}\label{thm:homogeneousWols2} \emph{(\cite[Remark]{ZhouCa2007})}
Let $d,s\in\N$. If $ds$ is even then
\begin{equation}\label{equ:wolHomoWtEven}
\zeta_{\calA_2}(\{s\}^d)=(-1)^{d-1} s \gb_{ds+1} p.
\end{equation}
If $ds$ is odd then
\begin{equation}\label{equ:wolHomoWtOdd}
\zeta_{\calA_3}(\{s\}^d)=(-1)^d\frac{s(ds+1)}{2}\gb_{ds+2} p^2.
\end{equation}
\end{thm}

For FESs, by using Euler numbers we obtained the following theorem in
\cite{TaurasoZh2010}.
\begin{thm} \label{thm:FESdepth1} \emph{(\cite[Corollary~2.3]{TaurasoZh2010})}
Let $s\in \N$. Then
\begin{alignat}{2}\label{equ:FESdepth1odd}
\zeta_{\calA_1}(\bar{s})=&
\left\{
  \begin{array}{ll}
    -2 q_2, & \  \\
   -2(1-2^{1-s}) \gb_{s},   & \
  \end{array}
\right.
& &\begin{array}{ll}
   \hbox{if $s=1$;} & \ \\
   \hbox{if $s>2$ is odd;}& \
  \end{array}
\\
\zeta_{\calA_2}(\bar{s})=&\,  s(1-2^{-s})  p \gb_{s+1}, &  &\ \text{if $s$ is even.}  \label{equ:FESdepth1even}
\end{alignat}
\end{thm}

The following results in depth 2 and 3 will be very useful.

\begin{thm}\label{thm:FMZVdepth2superbity1} \emph{(\cite[Theorem~3.1]{Zhao2008a},
\cite[Theorem~6.1]{Hoffman2004})}
For all positive integers $s$ and $t$
\begin{equation}\label{equ:depth2FMZV}
\zeta_{\calA_1}(s,t)=\zeta^\star_{\calA_1}(s,t)= (-1)^s \binom{s+t}{s}\gb_{s+t} .
\end{equation}
If $s,t\in\N$ and $w=s+t$ is \emph{odd}, then
\begin{equation}\label{equ:depth2FES}
\zeta_{\calA_1}(\bar{s},t)=\zeta_{\calA_1}(s,\bar{t})
=-\zeta^\star_{\calA_1}(\bar{s},t)=-\zeta^\star_{\calA_1}(s,\bar{t})
=(1-2^{1-w})\gb_w.
\end{equation}
\end{thm}

\begin{thm}\label{thm:FMZVdepth3superbity1} \emph{(\cite[Theorem~3.5]{Zhao2008a},
\cite[Theorem~6.2]{Hoffman2004})}
Let $(l,m,n)\in \N^3$. If $w=l+m+n$ is \emph{odd} then
\begin{equation*}
-\zeta_{\calA_1}(l,m,n)
=\zeta^\star_{\calA_1}(l,m,n)
=\left[(-1)^{l}\binom{w}{l}-(-1)^{n}\binom{w}{n}\right] \frac{\gb_{w}}{2}.
\end{equation*}
\end{thm}

\section{Double shuffle relations of finite Euler sums}\label{sec:FMZVdbsf}

\subsection{Stuffle relations of finite Euler sums}\label{sec:FMZVstuffle}
There are many $\Q$-linear relations among Euler sums. One of the most
important tools to study these is so-called (regularized)
double shuffle relations. However, in the finite setting, the
shuffle structure is not easily seen due to the lack of integral
expressions, although the stuffle is obvious. For instance,
for any positive integer $n$, we have
\begin{equation*}
H_n(s)H_n(t)=H_n(s,t)+H_n(t,s)+H_n(s\oplus t), \quad \forall s,t\in\db,
\end{equation*}

By extending an idea of Hoffman \cite{Hoffman2005} Racinet studied
the cyclotomic analogs of MZVs of level $N$ in \cite{Racinet2002}
using algebras of words. At level two, these analogs are
exactly the Euler sums considered in \cite{Zhao2010a}.

\begin{defn} \label{defn:fAlevelN}
Let the level $N$ be a positive integer ($N=1$ or 2 in this paper).
Let $\gG_N$ be the set of $N$th roots of unity.
The set of alphabet $\setX=\setX_N$ consists of $N+1$ letters $\x_\xi$
for $\xi\in\{0\}\cup \gG_N$. Let $X^*$ be the set of
words over $\setX$ (\emph{i.e.}, monomials in the letters in $\setX$)
including the \emph{empty word} $\myone$. The \emph{weight}
of a word $\bfw$ is the number of letters contained in $\bfw$ and its \emph{depth}
is the number of $\x_\xi$'s ($\xi\in\gG_N$) contained in $\bfw$.
Define the \emph{Hoffman-Racinet algebra} of level $N$, denoted by $\fA_N$,
to be the (weight) graded noncommutative polynomial $\Q$-algebra
generated by $X^*$.
Let $\fA_N^0$ be the subalgebra of $\fA_N$ generated by words not
beginning with $x_1$ and not ending with $x_0$. The words in $\fA_N^0$
are called \emph{admissible words.}

A shuffle product, denoted by $\sha$, is defined on $\fA_N$ as follows:
$\myone \sha\bfw=\bfw\sha\myone=\bfw$ for all
$\bfw\in \fA_N$, and
\begin{equation*}
x \bfu \sha y \bfv=
x(\bfu \sha y\bfv)
+y (x \bfu \sha \bfv),
\end{equation*}
for all $x,y\in\setX$ and $\bfu,\bfv\in\setX^*$. Then $\sha$ is
extended $\Q$-linearly over $\fA_N$.
\end{defn}

Let $\fA_N^1$ be the subalgebra of $\fA_N$ generated by
those words not ending with $\x_0$. For every $n\in\N$,
we define the weight $n$ element
\begin{equation*}
\y_{n,\xi}:=\x_0^{n-1}\x_\xi, \quad \xi\in\gG_N.
\end{equation*}
Then $\fA_N^1$ is clearly generated over the
alphabet $\setY_N:=\{\y_{n,\mu}: \  n\in\N,\ \mu\in\gG_N\}$.
Let $\setY_N^\ast$ be the set of words over $\setY_N$.
We now define a stuffle product $\ast$ on $\fA_N^1$
as follows: $\myone\ast\bfw=\bfw\ast\myone=\bfw$ for all
$\bfw\in\setY_N^\ast$, and
\begin{equation}\label{equ:worFormStuff}
\y_{m,\mu} \bfu \ast \y_{n,\nu} \bfv=
\y_{m,\mu}(\bfu \ast \y_{n,\nu}\bfv)
+\y_{n,\nu} (\y_{m,\mu} \bfu \ast \bfv)
+\y_{m+n,\mu\nu}(\bfu \ast \bfv),
\end{equation}
for all $m,n\in\N, \mu,\nu\in\gG_N$ and $\bfu,\bfv\in\setY_N^\ast$.
Then $\ast$ is extended $\Q$-linearly over $\fA_N^1$.

For convenience, we define
$W:\bigcup_{d\ge1}\db^d\to\setY_2^\ast$,
$W(\bfs)=\y_{|s_1|,\sgn(s_1)} \dots \y_{|s_d|,\sgn(s_d)}$
for all $\bfs=(s_1,\dots,s_d)\in \db^d.$
Clearly the map $W$ is a bijection.

\begin{defn}
Define $\zeta_{\ast}:\fA_2^0\to\R$ as follows:
for any admissible word $\bfw=W(\bfs)\in\fA_2^0$ where $\bfs\in\db^d$
we set $\zeta_{\ast}(\bfw):=\zeta(\bfs)$.
Then we extend it $\Q$-linearly to $\fA_2^0$.
\end{defn}

\begin{defn}
Define $\zeta_{\calA_\ell,\ast}:\fA_2^1\to\calA_\ell$ as follows:
for any word $\bfw=W(\bfs)\in\fA_2^1$ where $\bfs\in\db^d$
we set $\zeta_{\calA_\ell,\ast}(\bfw):=\zeta_{\calA_\ell}(\bfs)$.
Then we extend it $\Q$-linearly to $\fA_2^1$.
\end{defn}

\begin{prop}\label{prop:zetaStuffle}
The map $\zeta_{\ast}:(\fA_2^0,\ast)\to\R$
is an algebra homomorphism. So is the map
$\zeta_{\calA_\ell,\ast}: (\fA_2^1,\ast)\to\calA_\ell$.
\end{prop}
\begin{proof}
By induction on $|\bfu|+|\bfv|$ we can prove easily that
\begin{alignat*}{2}
\zeta_{\ast}(\bfu*\bfv)=&\, \zeta_{\ast}(\bfu)\zeta_{\ast}(\bfv) & \quad&
\forall \bfu,\bfv\in\fA_2^0, \\
 \zeta_{\calA_\ell,\ast}(\bfu*\bfv)=&\, \zeta_{\calA_\ell,\ast}(\bfu)\zeta_{\calA_\ell,\ast}(\bfv) & \quad&
\forall \bfu,\bfv\in\fA_2^1.
\end{alignat*}
We leave the details to the interested reader.
\end{proof}

\begin{defn}\label{defn:FESstuffle}
To find as many $\Q$-linear relations as possible in weight $w$
we may choose all the known relations in weight $k<w$, multiply
them by $\zeta_{\calA_\ell}(\bfs)$
for all $\bfs$ of weight $w-k$,  and then expand all the products
using the stuffle relation \eqref{equ:worFormStuff}. All the $\Q$-linear
relations among FESs of the same weight produced in this way are
called \emph{linear stuffle relations of FESs}.
\end{defn}

\begin{eg}\label{eg:stuffle}
By \eqref{equ:depth2FMZV} and \eqref{equ:depth2FES} we have
\begin{equation*}
     \zeta_{\calA_1}(2,1)=4\zeta_{\calA_1}(2,\bar1)=3\gb_3.
\end{equation*}
Multiplying $\zeta_{\calA_1}(\bar1)$ on both sides we get a
linear stuffle relation of FESs of weight 4:
\begin{multline}
\label{equ:missingWt4}
 \zeta_{\calA_1}(\bar1,2,1)+\zeta_{\calA_1}(\bar3,1)
+\zeta_{\calA_1}(2,\bar1,1)
+\zeta_{\calA_1}(2,1,\bar1)+\zeta_{\calA_1}(2,\bar2)\\
= 8\zeta_{\calA_1}(2,\bar1,\bar1)+4\zeta_{\calA_1}(\bar3,\bar1)
+4\zeta_{\calA_1}(\bar1,2,\bar1).
\end{multline}
\end{eg}

\subsection{Shuffle relations of FESs}\label{sec:FESshuffle}
It is not hard to see that Euler sums can be expressed by iterated integrals.
Suppose $s_j\in\db$ and $\sgn(s_j)=\mu_j$ for all $j=1,\dots,d$. Then
\begin{equation}\label{equ:EulerSumIteIntegral}
\zeta(s_1,\dots,s_n)=\int_0^1 \left(\frac{\ud t}{t}\right)^{|s_1|-1}\left(\frac{\ud t}{\xi_1-t}\right)\dots
\left(\frac{\ud t}{t}\right)^{|s_d|-1}\left(\frac{\ud t}{\xi_d-t}\right),
\end{equation}
where $\xi_i=\prod_{j=1}^i \mu_i$, $i=1,\dots,d$. We thus define
$\bfp,\bfp^{-1}:\fA_2^1 \to\fA_2^1$ by
\begin{equation*}
   \bfp(\y_{s_1,\xi_1}\cdots \y_{s_d,\xi_d})
   :=\y_{s_1,\mu_1}\cdots \y_{s_d,\mu_d},
\end{equation*}
where $\mu_i=\prod_{j=1}^i \xi_j$ and
\begin{equation*}
   \bfp^{-1}(\y_{s_1,\mu_1}\cdots \y_{s_d,\mu_d})
   :=\y_{s_1,\xi_1}\cdots \y_{s_d,\xi_d},
\end{equation*}
where $\xi_j=\mu_{j-1}^{-1}\mu_j$ (setting $\mu_0=1$).

For all $s_1,\dots,s_d\in\db$, we now define the
one-variable multiple polylog
\begin{equation*}
\singL_{s_1,\dots,s_d}(z):=\sum_{k_1>k_2>\ldots>k_d\ge 1}z^{k_1} \prod_{j=1}^d
\frac{\sgn (s_j)^{k_j}}{k_j^{|s_j|}}, \quad |z|<1.
\end{equation*}
Then it is easy to extend \eqref{equ:EulerSumIteIntegral} to these functions.
\begin{lem} \label{lem:RacProp228} \emph{(\cite[Proposition~2.2.8]{Racinet2002})}
For all $s_1,\dots,s_d\in\db$, we have
\begin{equation*}
\singL_{s_1,\dots,s_d}(z)=\int_0^z \left(\frac{\ud t}{t}\right)^{|s_1|-1}\left(\frac{\ud t}{\xi_1-t}\right)\dots \left(\frac{\ud t}{t}\right)^{|s_d|-1}\left(\frac{\ud t}{\xi_d-t}\right),
\end{equation*}
where $\xi_i= \sgn(s_{i-1})\sgn(s_i)$, $i=1,\dots,d$. Here we have set $\sgn(s_{0})=1$.
\end{lem}

Define the map $\singL(-;z):(\fA_2^1,\ast)\to\R[\![z]\!]$ by setting
$\singL(\bfw;z)=\singL_\bfs(z)$ for all $\bfw=W(\bfs)$. Then
define the map $\singL_\sha(-;z):(\fA_2^1,\sha)\to\R[\![z]\!]$ by setting
$\singL_\sha(\bfw;z)=\singL(\bfp(\bfw);z)$.

\begin{prop}\label{prop:SVMPLStuffle}
The map $\singL_\sha(-;z):(\fA_2^1,\sha)\to\R[\![z]\!]$ is an algebra homomorphism.
\end{prop}
\begin{proof}
This follows from Lemma~\ref{lem:RacProp228} and the shuffle relation
satisfied by the iterated integrals.
\end{proof}

When $z\to 1$ we obtain the algebra homomorphism
$\zeta_\sha:(\fA_2^0,\sha)\to\R$ with
$\zeta_\sha(\bfw)=\zeta_\ast(\bfp(\bfw))$.
One of the main results of \cite{Racinet2002} is the following theorem.
\begin{thm}\label{thm:2wayReg} \emph{(\cite[Proposition~2.4.14]{Racinet2002})}
Set $\zeta_\sha(\x_1)=\zeta_\ast(\x_1)=T$. Then
\begin{itemize}
  \item[\upshape{(i)}] $\zeta_\ast$ can be extended to an
algebra homomorphism $\zeta_\ast: (\fA_2^1,\ast)\to\R[T]$;

  \item[\upshape{(ii)}] $\zeta_\sha$ can be extended to an
algebra homomorphism $\zeta_\sha: (\fA_2^1,\sha)\to\R[T]$.
\end{itemize}
Moreover,
\begin{equation*}
     \zeta_\sha=\rho\circ\zeta_\ast\circ\bfp
\end{equation*}
where $\rho:\R[t]\to\R[t]$ is an $\R$-linear map such that
\begin{equation*}
 \rho(e^{Tu})=\exp\left(\sum_{n=2}^\infty
 \frac{(-1)^n}{n}\zeta(n)u^n\right)e^{Tu},\qquad |u|<1.
\end{equation*}
\end{thm}

\begin{eg}
Suppose $\bfs=(\bar2,1)$. Then $W(\bfs)=\y_{2,-1}\y_{1,1}$ and
\begin{equation*}
  \singL(\y_{2,-1}\y_{1,1};z)=\int_0^1 \frac{\ud t}{t}
    \left(\frac{-d t}{1+t}\right)^2= \singL_{\bar2,1}(z).
\end{equation*}
\end{eg}

Set $\zeta_{\calA_1,\sha}=\zeta_{\calA_1,\ast}\circ\bfp:\fA_2^1\to\calA_1$.
Although $\zeta_{\calA_1,\sha}$ is not an algebra homomorphism
from $(\fA_2^1,\sha)$ we will see in the next theorem that it does
provide a kind of shuffle relation. It is the FES analog of
\cite[Theorem~1.7]{Jarossay2014} for symmetrized Euler sums
(see \eqref{equ:SESsha} and \eqref{equ:SESast}).

Define $\gt:\setY_1^\ast\to \setY_1^\ast$ by
\begin{equation*}
\tau(\x_0^{s_1-1}\x_1 \cdots \x_0^{s_d-1}\x_1)=
(-1)^{s_1+\cdots+s_d} \x_0^{s_d-1} \x_1\cdots \x_0^{s_1-1} \x_1.
\end{equation*}

\begin{thm}\label{thm:FESshuffle}
For all words $\bfw,\bfu\in\setY_1^\ast$ and
$\bfv\in\setY_2^\ast$, we have
\begin{itemize}
 \item[\upshape{(i)}]
 $\zeta_{\calA_1,\sha}(\bfu \sha \bfv)=\zeta_{\calA_1,\sha}(\tau(\bfu)\bfv)$,

\item[\upshape{(ii)}]
 $\zeta_{\calA_1,\sha}( (\bfw\bfu) \sha \bfv)=\zeta_{\calA_1,\sha}(\bfu \sha\tau(\bfw)\bfv)$,

\item[\upshape{(iii)}] For all $s\in\N$,
 $\zeta_{\calA_1,\sha}( (\x_0^{s-1}\x_1 \bfu) \sha\bfv)
 =(-1)^s \zeta_{\calA_1,\sha}(\bfu \sha (\x_0^{s-1}\x_1\bfv)).$
\end{itemize}
\end{thm}

\begin{proof}
Taking $\bfu=\emptyset$ and then $\bfw=\bfu$ we see that (ii) implies (i).
Decomposing $\bfw$ into strings of the type $\x_0^{s-1}\x_1$
we see that (iii) implies (ii). Now we only need to prove (iii).

For simplicity let $\a=\x_0$ and $\b=\x_1$ in the rest of this proof.
Let $s_j\in\db$ with $\sgn(s_j)=\eta_j$ for $j=1,\dots,d$.
Let $\bfu=W(\bfs)$ and
$\bfv=\bfp(\bfu)=\y_{t_1,\xi_1} \dots \y_{t_d,\xi_d}\in\db^d$.
Then clearly $\bfp(\b\bfu)=\b\bfv$.

For any prime $p>2$, the coefficient of
$z^p$ in $\singL_\sha(\b\bfu;z)$ is given by
\begin{equation*}
  {\rm Coeff}_{z^p}\big[\singL_\sha(\b\bfu;z)\big]=
  \frac{1}{p}\sum_{p>k_1>\cdots>k_d>0}
    \frac{\xi_1^{k_1}\cdots\xi_d^{k_d}}{k_1^{t_1}\cdots k_d^{t_d}}
  =\frac{1}{p} H_{p-1}(\bfp(\bfu)).
\end{equation*}
Observe that
\begin{equation*}
\b \Big( (\a^{s-1}\b \bfu) \sha\bfv -(-1)^s   \bfu \sha (\a^{s-1}\b \bfv)\Big)
=\sum_{i=0}^{s-1} (-1)^{i} (\a^{s-1-i}\b \bfu) \sha (\a^{i}\b \bfv).
\end{equation*}
By first applying $\singL_\sha(-;z)$ to the above and then extracting
the coefficients of $z^p$  from both sides we get
\begin{align*}
&\, \frac{1}{p}\Big( H_{p-1}\circ\bfp\big((\a^{s-1}\b \bfu) \sha\bfv\big)
    -(-1)^s  H_{p-1}\circ\bfp\big(\bfu \sha (\a^{s-1}\b \bfv)\big)\Big)\\
=&\,\sum_{i=0}^{s-1} (-1)^{i} {\rm Coeff}_{z^p}
    \big[\singL_\sha(\a^{s-1-i}\b \bfu;z)\singL_\sha(\a^{i}\b \bfv;z)\big] \\
=&\,\sum_{i=0}^{s-1} (-1)^{i}\sum_{j=1}^{p-1}
{\rm Coeff}_{z^j}\big[\singL_\sha(\a^{s-1-i}\b \bfu;z) \big]
{\rm Coeff}_{z^{p-j}}\big[\singL_\sha(\a^{i}\b \bfv;z)\big]
\end{align*}
by Proposition~\ref{prop:SVMPLStuffle}. Now the last sum
is $p$-integral since $p-j<p$ and $j<p$. Therefore we get
\begin{equation*}
H_{p-1}\circ\bfp((\a^{s-1}\b \bfu) \sha\bfv)\equiv
(-1)^s  H_{p-1}\circ\bfp(\bfu \sha (\a^{s-1}\b \bfv)) \pmod{p}
\end{equation*}
which completes the proof of (iii).
\end{proof}

\begin{defn}\label{defn:FESdbsf}
A relation produced by Theorem~\ref{thm:FESshuffle} is called a
\emph{linear shuffle relation} of FES. For each weight $w\ge 2$,
by the \emph{double shuffle relations} of FESs of weight $w$ we
mean all the linear shuffle relations of weight $w$
and all the linear stuffle relations of $w$ defined in
Definition~\ref{defn:FESstuffle}.

Restricting to FMZVs, we obtain the
\emph{linear shuffle relations} and \emph{double shuffle relations}
of FMZVs.
\end{defn}

\section{Reversal relations}\label{sec:FMZVrev}
 For any $\bfs=(s_1,\dots,s_d)\in\db^d$, denote its \emph{reversal} by
$\tla{\bfs}=(s_d,\dots,s_1)$
and set $\sgn(\bfs)= \prod_{j=1}^d \sgn(s_j)$.
The following results are called
\emph{reversal relations}.
\begin{thm} \label{thm:reversal} \emph{(\cite[(6)]{Zhao2011c})}
Let $\bfs=(s_1,\dots,s_d)\in\db^d$. Then
\begin{equation}\label{equ:reversal}
\zeta_{\calA_1}(\tla{\bfs})= (-1)^{|\bfs|} \sgn(\bfs)\zeta_{\calA_1}(\bfs),\qquad
\zeta^\star_{\calA_1}(\tla{\bfs})= (-1)^{|\bfs|} \sgn(\bfs)\zeta^\star_{\calA_1}(\bfs).
\end{equation}
\end{thm}

Theorem~\ref{thm:reversal} can be lifted to superbity two.
\begin{thm} \label{thm:reversal2} \emph{(\cite[Theorem~2.1]{Zhao2008a})}
Let $\bfs=(s_1,\dots,s_d)\in\db^d$. Let $\bfe_i=(0,\dots,0,1,0,\dots,0)$ where $1$
appears at the $i$th component. Then
\begin{align}\label{equ:reversal2}
(-1)^{|\bfs|} \sgn(\bfs)\zeta_{\calA_2}(\tla{\bfs})=\, & \zeta_{\calA_2}(\bfs)+
p \sum_{i=1}^d |s_i|\zeta_{\calA_2}(\bfs\oplus \bfe_i),\\
 (-1)^{|\bfs|} \sgn(\bfs)\zeta^\star_{\calA_2}(\tla{\bfs})=\, &\zeta^\star_{\calA_2}(\bfs)+
p \sum_{i=1}^d |s_i|\zeta^\star_{\calA_2}(\bfs\oplus \bfe_i),\label{equ:reversal2star}
\end{align}
where the binary operation $\oplus$ is carried out componentwise
by \eqref{equ:oplusDefn}.
\end{thm}

\section{Quasi-symmetric functions with signed powers}\label{sec:QSym}
To derive more relations between MHSs and AMHSs we turn to
the theory of quasi-symmetric functions. To treat AMHSs with signed indices
in $\db$ we have to allow the powers in these quasi-symmetric functions
to be signed numbers.
\begin{defn} \label{defn:CQSym}
We denote by $\Z[x_1,\dots, x_n;\db]$ the set of polynomials in $x_1,\dots, x_n$
with signed powers, namely,
\begin{equation*}
 \Z[x_1,\dots, x_n;\db]
:=\left\{\left.\sum_{e_1=\ol{d_1}}^{d_1}\cdots \sum_{e_n=\ol{d_n}}^{d_n} c_{e_1,\dots,e_n}
    x_1^{e_1} \cdots  x_n^{e_n}\right| d_1,\dots,d_n\in \N_0, c_{e_1,\dots,e_n}\in \Z\right\}.
\end{equation*}
Here, we set $\ol{0}=0$ and $\sum_{e=\ol{d}}^{d}$ means $e$ runs through
the set $\{\ol{d},\dots,\ol{1},0,1,\dots,d\}$.
Furthermore, $x_j^e x_j^{e'}=x_i^{e\oplus e'}$ for any $j\le n$
and $e,e'\in\db$. Also we set
\begin{equation*}
\deg(x_1^{e_1} \cdots  x_n^{e_n})=|e_1|+\dots+|e_n|.
\end{equation*}
\end{defn}

\begin{defn}\label{defn:QSym2}
Let $\bfx=(x_j)_{j\ge 1}$.
An element of finite degree $F(\bfx)$ in $\Z[\![\bfx;\db]\!]$  is called
a \emph{quasi-symmetric function}
if for any $i_1>i_2>\cdots>i_d$ and $j_1>j_2>\cdots>j_d$ and any signed powers
$e_1,\dots,e_d\in \db$ the coefficients of the monomials
$x_{i_1}^{e_1}\cdots x_{i_d}^{e_d}$ and $x_{j_1}^{e_1}\cdots x_{j_d}^{e_d}$
are the same. The set of all such quasi-symmetric functions is denoted by $\QSym_2$.
\end{defn}

For positive integer $n$ we define a (weight) graded algebra homomorphism
\begin{align*}
  \phi_n: \quad(\fA_2^1,\ast)   & \lra \quad \QSym_2 \\
       W(\bfs) & \lmaps
       \sum_{n\ge k_1>k_2>\dots>k_d\ge 1} x_{k_1}^{s_1}\cdots x_{k_d}^{s_d},\quad \forall d\le n, \bfs\in\db^d,
\end{align*}
and set $\phi_n(\myone)=1$ and
$\phi_n(\bfw)=0$ if the $|\bfw|>n$. It is easy to
see we can make $(\phi_n)_{n\ge 1}$ into a compatible system to
obtain a homomorphism $\phi:\fA_2^1\to\QSym_2$.

\begin{eg}
We have $\phi(\y_{1,1}\y_{2,-1})=\sum_{i>j\ge 1} x_ix_j^{\bar 2}\in\QSym_2$ but it is \emph{not}
a symmetric function since the monomial $x_2x_1^{\bar 2}$ appears but
$x_1x_2^{\bar 2}$ does not.
\end{eg}

It is not hard to see that an integral basis for $\QSym_2$ can be chosen as
\begin{align}
E_\bfs= E_{s_1,\dots,s_d}&\, := \sum_{k_1\ge k_2\ge\dots\ge k_d }
    x_{k_1}^{s_1}\cdots x_{k_d}^{s_d}, \quad\text{or} \notag \\
M_\bfs= M_{s_1,\dots,s_d}&\, := \sum_{k_1>k_2>\dots>k_d }
    x_{k_1}^{s_1}\cdots x_{k_d}^{s_d}
    =\phi(W(s_1,\dots,s_d)). \label{equ:transportQSymFromfA}
\end{align}
This yields the following theorem which can be compared to
\cite[Theorem~2.2]{Hoffman2004}.
\begin{thm}
$\phi$ provides an isomorphism $(\fA_2^1,*)\cong\QSym_2$.
\end{thm}
\begin{proof}
Clear.
\end{proof}

\begin{thm} \label{thm:HoffmanAntipodeS}
The antipode $S$ of $\QSym_2$ is given by the followings:
for every $\bfs\in\db^d$,
\begin{itemize}
\item[\upshape{(i)}]
    $S(M_\bfs)=(-1)^d E_{\revs{\bfs}}$, where
    $E_\bft=\sum_{\bfs\preceq \bft} M_\bfs$
    and $\bft\preceq \bfs$ means $\bft$ can be obtained
    from $\bfs$ by combining some of its parts using $\oplus$.
\item[\upshape{(ii)}]
    $\displaystyle S(M_\bfs)=(-1)^d E_{\revs{\bfs}}
        =\sum_{\bigsqcup_{j=1}^r \bfs_j=\bfs}
     (-1)^r M_{\bfs_1}M_{\bfs_2}\cdots M_{\bfs_r}$, where
   $\bigsqcup_{j=1}^r \bfs_j$ is the concatenation of $\bfs_1$ to $\bfs_r$.
\item[\upshape{(iii)}]
    $\displaystyle M_{\revs{\bfs}} =(-1)^d \sum_{\bigsqcup_{j=1}^r \bfs_j=\bfs}
     (-1)^r E_{\bfs_1}E_{\bfs_2}\cdots E_{\bfs_r}$.
\end{itemize}
\end{thm}
\begin{proof}
The case for positive compositions $\bfs$ for part (i) and (ii)
is just \cite[Theorem~6.2]{Hoffman2005}.
A careful reading of their proofs reveals that all the steps are still valid when
the components of $\bfs$ are signed number.
Applying antipode $S$ to (ii) and use (i) we can derive (iii) quickly.
\end{proof}

\section{Concatenation relations of FESs}\label{sec:concat}
We now apply the above results concerning quasi-symmetric
functions to FESs. In order to do so, for any $n\in\N$, we define
the algebra homomorphism $\rho_n:(\fA_2^1,*)\to\Q$ such that
\begin{equation}\label{equ:defnOfRho}
  \rho_n(M_\bfs)=H_n(\bfs).
\end{equation}

\begin{thm} \label{thm:Hoff}
For all $\bfs=(s_1,\dots,s_d)\in\db^d$ and positive integer $\ell$
\begin{equation}  \label{equ:FMZVSH}
\zeta^\star_{\calA_\ell}(\bfs) =\sum_{\bft\preceq \bfs} \zeta_{\calA_\ell}(\bft),
\qquad \zeta_{\calA_\ell}(\bfs) =\sum_{\bft\preceq \bfs} (-1)^{\dep(\bfs)-\dep(\bft)}\zeta^\star_{\calA_\ell}(\bft).
\end{equation}
When $\ell=1$ we have
\begin{align} \label{equ:FMZSV=concatFMZV}
\zeta^\star_{\calA_1}(\tla{\bfs}) &=(-1)^d \sum_{\bigsqcup_{j=1}^r
\bfs_j=\bfs} (-1)^r \prod_{j=1}^r \zeta_{\calA_1}(\bfs_j),\\
\label{equ:FMZV=concatFMZSV}
\zeta_{\calA_1}(\tla{\bfs}) &=(-1)^d \sum_{\bigsqcup_{j=1}^r
\bfs_j=\bfs} (-1)^r \prod_{j=1}^r \zeta^\star_{\calA_1}(\bfs_j).
\end{align}
\end{thm}

\begin{proof}
These equations follow from the definition of $E_\bfs$,
Theorem~\ref{thm:HoffmanAntipodeS}(ii) and (iii), respectively,
after we apply $\rho_{p-1}$ for all primes $p$.
\end{proof}

\begin{defn}
We will call the relations in \eqref{equ:FMZSV=concatFMZV}
and \eqref{equ:FMZV=concatFMZSV} the
\emph{concatenation relations} between FMZVs and FESs.
\end{defn}

\begin{eg}
By \eqref{equ:FMZVSH}
and the concatenation relation \eqref{equ:FMZSV=concatFMZV} we have
\begin{equation*}
 \zeta_{\calA_1}(1,\bar2)+\zeta_{\calA_1}(\bar3)= \zeta^\star_{\calA_1}(1,\bar2)
=\zeta_{\calA_1}(\bar2)\zeta_{\calA_1}(1)-\zeta_{\calA_1}(\bar2,1).
\end{equation*}
Thus we get $\zeta_{\calA_1}(\bar3)=-2\zeta_{\calA_1}(\bar2,1)$ since
$\zeta_{\calA_1}(1)=0$ and $\zeta_{\calA_1}(1,\bar2)=\zeta_{\calA_1}(\bar2,1)$
by the reversal relation \eqref{equ:reversal}.
\end{eg}

\section{Duality of FMZVs}\label{sec:duality}
The FMZVs satisfy a different kind of duality from that of MZVs.

\begin{defn}\label{defn:v-dual}
For positive integers $r_1,\dots,r_\ell,t_1,\dots,t_\ell$, let
\begin{equation*}
\bfs=(r_1,\{1\}^{t_1-1},r_2+1,\{1\}^{t_2-1},\dots,r_\ell+1,\{1\}^{t_\ell-1}).
\end{equation*}
We define the v-dual of $\bfs$ by
\begin{equation}\label{equ:v-dualDefn2}
\bfs^\vee=(\{1\}^{r_1-1},t_1+1,\{1\}^{r_2-1},t_2+1,\dots,t_{\ell-1}+1,\{1\}^{r_\ell-1},t_\ell).
\end{equation}
\end{defn}
{}From the definition we clearly have
\begin{equation}\label{equ:lengthOfV-dual}
  \dep(\bfs)+\dep(\bfs^\vee)=|\bfs|+1.
\end{equation}

The v-dual can be easily explained using the conjugation on
the ribbons (a kind of skew-Young diagrams) as follows.
For a composition $\bfs=(s_1,\dots,s_d)$ the ribbon $R_\bfs$ is
defined to be the skew-Young diagram of $d$ rows whose $j$th row
starts below the last box of $(j-1)$st row and has
exactly $s_j$ boxes. Recall that the conjugate of a (skew-)Young diagram
is the mirror image about the diagonal line going from the south-west
corner to north-east.
For e.g., the following two diagrams give the ribbon
$R_{1,3,2}$ and its conjugate:
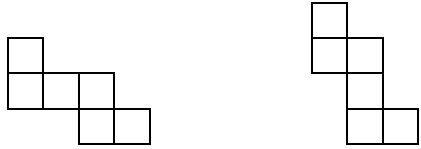
\begin{center}
 \young({\ },{\ }{\ }{\ },::{\ }{\ })
 \hskip2cm
 \young({\ },{\ }{\ },:{\ },:{\ }{\ })
\end{center}
In general it can be shown without too much difficulty that
the ribbon $R_{\bfs^\vee}$ is exactly the conjugate of the
ribbon $R\!\text{\raisebox{-1pt}{$\scriptscriptstyle \revs{\bfs}$} }\hskip-0.3pt$.
So $(2,3,1)^\vee=(1,2,1,2)$.

\begin{thm}\label{thm:dualFMZVsuperb=1} \emph{(\cite[Theorem~6.7]{Hoffman2005})}
 Let $\bfs=(s_1,\dots,s_d)\in\N^d$. Then
\begin{equation}  \label{equ:FMZVduality}
\zeta^\star_{\calA_1}(\bfs) =\, -\zeta^\star_{\calA_1}(\bfs^\vee).
\end{equation}
\end{thm}

We now consider duality property in superbity two.
\begin{thm}\label{thm:FMZVv-dualsuperb=2}  \emph{(\cite[Theorem~2.11]{Zhao2008a})}
Let $\bfs$ be any composition of positive integers of weight $w$.  Then
\begin{equation} \label{equ:dualityLevel2}
-\zeta_{\calA_2}^\star(\bfs^\vee)=\zeta_{\calA_2}^\star(\bfs)
+p\cdot\sum_{\bft\preceq\bfs}\zeta_{\calA_2}(1,\bft).
\end{equation}
\end{thm}

Parallel to \cite[Corollary~1.12]{Jarossay2014} for the symmetrized
MZVs (see \eqref{equ:SMZVsha} and \eqref{equ:SMZVast}), the
following result on another kind of duality provides further
evidence for Conjecture~\ref{conj:KanekoZagier}.
\begin{thm}\label{thm:FMZVphi-dual}  \emph{(\cite[Theorem~6.9]{Hoffman2005})}
For all $\bfw\in\fA_1^1$ we have
\begin{equation*}
   \zeta_{\calA_1}(\bfw)=\zeta_{\calA_1}(\varphi(\bfw)),
\end{equation*}
where $\varphi$ is the involution on $\fA_1^1$ defined by
$\varphi(\x_0)=\x_0+\x_1$ and $\varphi(\x_1)=-\x_1$.
\end{thm}

\section{Dimension conjectures for (finite) MZVs}  \label{sec:dimConjMZV}
Denote by $\FMZ_{w,\ell}$ the $\Q$-vector subspace of $\calA_\ell$
generated by all FMZVs of weight $w$ and superbity $\ell$.
Further we write $\FMZ_w=\FMZ_{w,1}$.
Numerical evidence supports the following conjecture.

\begin{conj}\label{conj:FMZVsuperbity1}
Let $w$ be any positive integer.
\begin{itemize}
  \item [\upshape{(i)}] Set $d_0=1$ and $d_w=\dim \FMZ_w$
for all $w\ge 1$.
Then $d_1=d_2=0$ and
$$d_w=d_{w-2}+d_{w-3}\quad \forall w\ge 3.$$

  \item [\upshape{(ii)}] For all $w\ge3$, $\FMZ_w$ has a basis
$$\{\zeta_{\calA_1}(1,2,a_1,\dots,a_r): a_1,\dots,a_r=2 \text{ or }3, a_1+\dots+a_r=w-3\}.$$

  \item [\upshape{(iii)}] All the $\Q$-linear relations among FMZVs
can be produced by the double shuffle relations.
\end{itemize}
\end{conj}

In the spring of 2013, after the author gave an Ober-Seminar
talk at the Max Planck Institute for Mathematics at Bonn,
Prof.\ D.\ Zagier mentioned that he had come to the dimension part
of Conjecture~\ref{conj:FMZVsuperbity1} some time earlier \cite{Zagier2011}.
In fact, Kaneko and Zagier proposed the following more
precise relation between FMZVs and MZVs.
\begin{conj}\label{conj:KanekoZagier}
There is an $\Q$-algebra isomorphism
\begin{align*}
f_{\rm KZ}:\FMZ_{w,1}  & \longrightarrow \MZV_w/\zeta(2)\MZV_{w-2}\\
\zeta_{\calA_1}(\bfs)& \longmapsto \zeta_\sha^\Sy(\bfs)
\end{align*}
where for $\bfs=(s_1, \ldots, s_d)$, the symmetrized MZVs
\begin{align}
\zeta_\sha^\Sy(\bfs)=&\, \sum_{i=0}^d
 (-1)^{s_1+\cdots+s_i} \zeta_\sha(s_i,\dots,s_1) \zeta_\sha(s_{i+1},\dots,s_d), \label{equ:SMZVsha}\\
\zeta_\ast^\Sy(\bfs)=&\, \sum_{i=0}^d
 (-1)^{s_1+\cdots+s_i} \zeta_\ast(s_i,\dots,s_1) \zeta_\ast(s_{i+1},\dots,s_d),  \label{equ:SMZVast}
\end{align}
where $\zeta_\sha$ and $\zeta_\ast$ on the right-hand side are
given by Theorem~\ref{thm:2wayReg}.
\end{conj}

The following results are contained in \cite{Jarossay2014}
(see Remarque 1.3 and Fait 1.8).
\begin{prop}\label{prop:Fait1.8}
For all composition $\bfs$ of positive integers,
\begin{itemize}
  \item $\zeta_\sha^\Sy(\bfs)$ and $\zeta_\ast^\Sy(\bfs)$ are all finite, and
  \item $\zeta_\sha^\Sy(\bfs)-\zeta_\ast^\Sy(\bfs)\in\zeta(2)\MZV_{w-2}$ for all $|\bfs|=w$.
\end{itemize}
\end{prop}
Therefore, one may also replace $\zeta_\ast$ by $\zeta_\sha$ in
Conjecture~\ref{conj:KanekoZagier}. Further, the conjectured map
is surjective according to the next theorem proved by
Yasuda \cite{Yasuda2014}.

\begin{thm} \label{thm:SMZVgenerateMZV}
Let $\sharp=\sha$ or $\ast$. Then the space $\MZV_w$ is
generated by symmetrized MZVs
$\big\{\zeta^\Sy_\sharp(\bfs): |\bfs|=w\big\}$.
\end{thm}

However, even in weight 5, it seems
impossible to prove the map $f_{\rm KZ}$ is well defined and injective
at the moment.Indeed, we have
$$\zeta_{\calA_1}(4,1)=5\gb_5=(B_{p-5})_p.$$
But we know $\zeta_\sha^\Sy(4,1)=5\zeta(5)-2\zeta(2)\zeta(3)$
is conjecturally nonzero on the right-hand side so that if
$f_{\rm KZ}$ is well defined then $\gb_5\ne 0$ which
would imply that $B_{p-5}\not\equiv 0\pmod{p}$
for infinitely many primes $p$, a statement far from proved.
On the other hand, even if we could prove $\gb_5\ne 0$ by
some other means, we still don't know whether
$\zeta(2)\zeta(3)/\zeta(5)\in\Q$ is true or not,
thus we still don't know whether $f_{\rm KZ}$ is injective.

As a further support of Conjecture~\ref{conj:FMZVsuperbity1},
\begin{thm}
For all $3\le w\le 13$, we have
\begin{align*}
 \FMZ_w&\, =\big\langle \zeta_{\calA_1}(1,2,a_1,\dots,a_r):
    a_1,\dots,a_r\in\{2,3\}, a_1+\cdots+a_r=w-3 \big\rangle.
\end{align*}
Moreover, all the the relations in these weights
can be proved by the double shuffle relations
of FMZVs defined as in Definition~\ref{defn:FESdbsf}.
\end{thm}
\begin{proof}
This theorem can be proved with the help of Maple. So we leave
it to the interested  reader.
\end{proof}

We can also obtain the following detailed results by Maple computation.
Note that all depth one or two values are given by
Theorems~\ref{thm:homogeneousWols1} and \ref{thm:FMZVdepth2superbity1}.
All FMZVs of weight 6
and depth at least 4 can be computed from the next
theorem by the duality relations in Theorem~\ref{thm:dualFMZVsuperb=1}.

\begin{thm}\label{thm:wt6FMZV}
The FMZVs of superbity $1$ in depth 3 and weight 6
are given by Table~\ref{Table:wt6FMZV}
(or can be obtained by the reversal relations). Moreover,
for all the values in the table,
$\zeta^\star_{\calA_1}=-\zeta_{\calA_1}$ if weight and depth
has the same parity while $\zeta^\star_{\calA_1}=\zeta_{\calA_1}$ otherwise.
\begin{table}[!h]
{
\begin{center}
\begin{tabular}{  ||c|c||c|c||c|c|| } \hline
$\zeta_{\calA_1}(1,3,2)$ &$-{\small \frac{9}{2}}\gbb_{3}^2$ &
$\zeta_{\calA_1}(1,4,1)$ &  $ 3\gbb_{3}^2$ &
$\zeta_{\calA_1}(1,1,4)$ & $-{\small \frac{3}{2}}\gbb_{3}^2$\\ \hline
$\zeta_{\calA_1}(2,1,3)$ &  $ {\small \frac{3}{2}}\gbb_{3}^2$ &
$\zeta_{\calA_1}(2,3,1)$ & $-{\small \frac{9}{2}}\gbb_{3}^2$ &
$\zeta_{\calA_1}(3,1,2)$ & $ {\small \frac{3}{2}}\gbb_{3}^2$ \\ \hline
$\zeta_{\calA_1}(3,2,1)$ & $ 3\gbb_{3}^2$ &
$\zeta_{\calA_1}(4,1,1)$ & $-{\small \frac{3}{2}}\gbb_{3}^2$ &
$\zeta_{\calA_1}(1,2,3)$ &  $ 3\gbb_{3}^2$\\ \hline
\end{tabular}
\end{center}
}
\caption{FMZVs of superbity 1, depth 3 and weight 6.}
\label{Table:wt6FMZV}
\end{table}
\end{thm}

Note that all depth 3 odd weight FMZVs are given by
Theorems~\ref{thm:FMZVdepth3superbity1}. All FMZVs of weight 7
and depth at least 5 can be converted to FMZVs of depth at
most 3 by the duality relations in Theorem~\ref{thm:dualFMZVsuperb=1}.

\begin{thm}\label{thm:wt7FMZV}
The FMZVs of superbity 1, depth 4 and weight 7 are given by
Table~\ref{Table:wt7FMZV} (or can be obtained by the reversal relations).
Moreover, $\zeta^\star_{\calA_1}=\zeta_{\calA_1}$
for all the values in the table.
\begin{table} [!h]
{
\begin{center}
\begin{tabular}{  ||c|c||c|c||c|c|| } \hline
 $\zeta_{\calA_1}(1,1,1,4)$ & $-27\gb_7'$ &
 $\zeta_{\calA_1}(1,1,2,3)$ & $69\gb_7'$ &
 $\zeta_{\calA_1}(1,1,3,2)$ & $-27\gb_7'$ \\ \hline
 $\zeta_{\calA_1}(1,1,4,1)$ & $33\gb_7'$ &
 $\zeta_{\calA_1}(1,2,1,3)$ & $-27\gb_7'$ &
 $\zeta_{\calA_1}(1,2,2,2)$ & $-27\gb_7'$ \\ \hline
 $\zeta_{\calA_1}(1,2,3,1)$ & $-63\gb_7'$ &
 $\zeta_{\calA_1}(1,3,1,2)$ & $-9\gb_7'$ &
 $\zeta_{\calA_1}(2,1,1,3)$ & $33\gb_7'$ \\ \hline
 $\zeta_{\calA_1}(2,1,2,2)$ & $-63\gb_7'$ & \ &\  & \ &\ \\ \hline
\end{tabular}
\end{center}
}
\caption{FMZVs of superbity 1, depth 4 and weight 7
($\gb_7'=\gbb_7/16$).}
\label{Table:wt7FMZV}
\end{table}
\end{thm}

\section{Finite multiple zeta values of small superbities}
We now move up to superbity 2 and beyond, namely, we consider
congruences modulo $p$-powers.
First, we improve on Theorem~\ref{thm:FMZVdepth2superbity1}.
\begin{thm} \label{thm:FMZVdepth2superbity2} \emph{(\cite[Theorem~3.2]{Zhao2008a})}
Suppose $s,t\in\N$ have the same same parity. Then
\begin{align*}
\zeta_{\calA_2}(s, t) =\, & p
\left[(-1)^t s \binom{s+t+1}{t}-(-1)^t t\binom{s+t+1}{s}
 -s-t\right] \frac{\gb_{t+s+1}}{2}, \\
\zeta^\star_{\calA_2}(t, s) =\, & p
\left[(-1)^t s \binom{s+t+1}{t}-(-1)^t t\binom{s+t+1}{s}
 +s+t\right] \frac{\gb_{t+s+1}}{2} .
\end{align*}
\end{thm}

So up to weight 4 we have the following table
(see \cite[Proposition~3.7]{Zhao2008a}
for detailed computation in weight 4)
\begin{table}[!h]
{
\begin{center}
\begin{tabular}{  ||c|c||c|c||c|c|| } \hline
 $\zeta_{\calA_2}(2)$ & $2p\gb_{3}$ &
 $\zeta_{\calA_2}(1,1)$ & $- p\gb_{3}$ &
 $\zeta_{\calA_2}(1,2)$ & $\zeta_{\calA_2}(1,2)$ \\ \hline
 $\zeta_{\calA_2}(2,1)$ & $-\zeta_{\calA_2}(1,2)$ &
 $\zeta_{\calA_2}(4)$ & $4 p\gb_{5} $ &
 $\zeta_{\calA_2}(1,3)$ & ${\small \frac12}p\gb_{5} $\\ \hline
 $\zeta_{\calA_2}(3,1)$ & $-{\small \frac92}p\gb_{5}$ &
 $\zeta_{\calA_2}(1,1,2)$ & $3\gb_{5}$  &
 $\zeta_{\calA_2}(1,2,1)$ & $-{\small \frac92}p\gb_{5}$ \\ \hline
 $\zeta_{\calA_2}(2,1,1)$ & ${\small \frac{11}2}\gb_{5}$  &
 $\zeta_{\calA_2}(1,1,1,1)$ & $- \gb_{5}$  &\  & \\ \hline
\end{tabular}
\end{center}
}
\caption{FMZV of superbity 2 in weight up to 4.}
\label{Table:wt2-4FMZVsuperbity2}
\end{table}
Note that $\zeta_{\calA_2}(1)=\zeta_{\calA_2}(3)=\zeta_{\calA_2}(5)
=\zeta_{\calA_2}(\{1\}^3)=\zeta_{\calA_2}(\{1\}^5)=0.$
For larger weights we have the next two theorems.
\begin{thm} \label{thm:wt5FMZVleve2IDs}
In weight 5 we have
\begin{align}
\zeta^\star_{\calA_2}(1,3,1)=\zeta_{\calA_2}(1,3,1)=\, & 0,\label{equ:wt5superbity2z131}\\
\zeta^\star_{\calA_2}(2,1,2)=\zeta_{\calA_2}(2,1,2)=\, & -3 p \gb_{3}^2,\label{equ:wt5superbity2z212}\\
\zeta^\star_{\calA_2}(2, 3)=\zeta_{\calA_2}(2, 3)
=\, & 2\zeta^\star_{\calA_2}(4, 1)= 2\zeta_{\calA_2}(4, 1), \label{equ:wt5superbity2z32}\\
2\zeta^\star_{\calA_2}(2, 2, 1)-3\zeta_{\calA_2}(4,1)=\, & 9p\gb_{3}^2, \label{equ:wt5superbity2z122}
\end{align}
and
\begin{equation}\label{equ:wt5superbity2z113}
2\zeta^\star_{\calA_2}(1,1,3)=-2\zeta^\star_{\calA_2}(3,1,1)
=-2\zeta_{\calA_2}(1,1,3)= 2\zeta_{\calA_2}(3,1,1)= \zeta_{\calA_2}(4, 1).
\end{equation}
\end{thm}
\begin{proof}
Throughout this proof all congruences are taken modulo $p^2$. First,
by applying \eqref{equ:dualityLevel2} to $\bfs=(1,1,3)$ we get
\begin{multline}\label{equ:deriveHstar131}
-\zeta^\star_{\calA_2}(3,1,1)=\zeta^\star_{\calA_2}(1,1,3)
+p\Big(\zeta_{\calA_1}(1,1,1,3)+\zeta_{\calA_1}(1,2,3)\\
+\zeta_{\calA_1}(1,1,4)+\zeta_{\calA_1}(1,5)\Big)
=\zeta^\star_{\calA_2}(1,1,3),
\end{multline}
which yields the first ``='' in \eqref{equ:wt5superbity2z113}
by Theorem~\ref{thm:wt6FMZV}.
Notice by the reversal relations in Theorem~\ref{thm:reversal2} $\zeta^\star_{\calA_2}(1,3,1)=\zeta_{\calA_2}(1,3,1)$
and
\begin{equation*}
-\zeta_{\calA_2}(1,1,3)-\zeta_{\calA_2}(3,1,1)=p(\zeta_{\calA_1}(2,1,3)+\zeta_{\calA_1}(1,2,3)+3\zeta_{\calA_1}(1,1,4))
\end{equation*}
by Theorems~\ref{thm:wt6FMZV} and \ref{thm:wt7FMZV}.
Thus \eqref{equ:deriveHstar131} together with stuffle relations yields
\begin{multline*}
  0= \zeta_{\calA_2}(1,1)\zeta_{\calA_2}(3)
  =\zeta_{\calA_2}(1,3,1)+\zeta_{\calA_2}(1,1,3) \\
  +\zeta_{\calA_2}(3,1,1)+\zeta_{\calA_2}(2)\zeta_{\calA_2}(3)
    -\zeta_{\calA_2}(5)=\zeta_{\calA_2}(1,3,1).
\end{multline*}
which gives \eqref{equ:wt5superbity2z131} using Theorem~\ref{thm:reversal2}.
Applying \eqref{equ:dualityLevel2} to $\bfs=(2,1,2)$ we get
\begin{multline*}
-\zeta^\star_{\calA_2}(1,3,1)=
p\Big(\zeta_{\calA_1}(1,2,1,2)+\zeta_{\calA_1}(1,3,2)+\zeta_{\calA_1}(1,2,3)+\zeta_{\calA_1}(1,5)\Big)\\
+\zeta^\star_{\calA_2}(2,1,2)
=\zeta^\star_{\calA_2}(2,1,2)+3p \gb_{3}^2
\end{multline*}
using Table~\ref{Table:wt6FMZV} on page \pageref{Table:wt6FMZV}.
Thus \eqref{equ:wt5superbity2z212} follows immediately from \eqref{equ:wt5superbity2z131}.

Turning to the proof of \eqref{equ:wt5superbity2z32} we note that for any positive integer $k<p$ (setting $h_k=H_{k-1}$)
\begin{equation}\label{equ:expansionBinomp}
\begin{split}
& \frac{\binom{p-1}{k}}{\binom{p-1+k}{k}}=\frac{(p-1)(p-2)\cdots(p-k)}{p(p+1)\cdots(p+k-1)}
=\frac{(-1)^k k}{p}\prod_{j=1}^k\left(1-\frac{p}{j}\right)\prod_{j=1}^{k-1}\left(1+\frac{p}{j}\right)^{-1} \\
\equiv &\frac{(-1)^k k}{p}\Big(1-pH_k(1)+p^2H_k(1,1)\Big)\Big(1-ph_k(1)+p^2h_k(2)+p^2h_k(1,1)\Big)  \\
\equiv &\frac{(-1)^k k}{p}\bigg(1-2ph_k(1)-\frac{p}{k}
    +\frac{2p^2}{k}h_k(1)+2p^2h_k(2)+4p^2h_k(1,1)\bigg)  .
\end{split}
\end{equation}
By \cite[Theorem~2.1]{HessamiPilehrood2Ta2013} we obtain
(with abuse of notation $H=H_{p-1}$)
{\allowdisplaybreaks
\begin{align*}
& H^\star_{p-1}(\{2\}^a,3,\{2\}^b)\\
\equiv& -\frac{2}{p}\sum_{k=1}^{p-1}\frac{1}{k^{2b+2a+2}}\left(1-2ph_k(1)-\frac{p}{k}
    +\frac{2p^2}{k}h_k(1) +2p^2h_k(2)+4p^2h_k(1,1)\right) \\
-&\frac{4}{p}\sum_{k=1}^{p-1}\frac{h_k(2b+1)}{k^{2a+1}}\left(1-2ph_k(1)-\frac{p}{k}
    +\frac{2p^2}{k}h_k(1) +2p^2h_k(2)+4p^2h_k(1,1)\right) \\
=&-\frac{2}{p}H(2b+2a+2)-\frac{4}{p}H(2a+1,2b+1)+4H(2b+2a+2,1) \\
+&2H(2b+2a+3)+8H(2a+1,1,2b+1)+8H(2a+1,2b+1,1)  \\
+&8H(2a+1,2b+2)+4H(2a+2,2b+1)\\
-&4p\Big[H(2a+2b+3,1)+H(2a+2b+2,2)+2H(2a+2b+2,1,1)]   \\
-&8p\Big[H(2a+2,1,2b+1)+H(2a+2,2b+1,1)+H(2a+2,2b+2)\\
&\phantom{-} +H(2a+1,2,2b+1)+H(2a+1,2b+1,2)+H(2a+1,2b+3)\\
&\phantom{-} +2H(2a+1,1,2b+2)+2H(2a+1,2b+2,1)
+2H(2a+1,1,1,2b+1)\\
&\phantom{-} +2H(2a+1,1,2b+1,1)+2H(2a+1,2b+1,1,1)\Big].
\end{align*}}{}\noindent
Taking $a=1$ and $b=0$ we get
\begin{multline}\label{equ:H23midStep}
H^\star_{p-1}(2,3)=-\frac{2}{p}H(4)-\frac{4}{p}H(3,1)+16H(3,1,1)+8H(3,2)+8H(4,1)  \\
-4p\Big[H(5,1)+3H(4,2)+6H(4,1,1)+6H(3,2,1)+6H(3,1,2)+12H(3,1,1,1)\Big].
\end{multline}
By stuffle relation and using \eqref{equ:wt5superbity2z131} we see that
\begin{equation*}
   2\zeta_{\calA_2}(3,1,1)+\zeta_{\calA_2}(3,2)+\zeta_{\calA_2}(4,1)
   =\zeta_{\calA_2}(3,1)\zeta_{\calA_2}(1)-\zeta_{\calA_2}(1,3,1)= 0.
\end{equation*}
On the other hand
\begin{multline*}
\frac{1}{p}\Big(H(4)+2H(3,1)\Big)=\frac{1}{p}\Big(H(1)H(3)+H(3,1)-H(1,3)\Big)\\
 \equiv -\Big(3H(4,1)+H(3,2)+3pH(4,2)+6pH(5,1)\Big)
\end{multline*}
by \eqref{equ:reversal2}. Thus \eqref{equ:H23midStep} is reduced to
\begin{multline*}
H^\star_{p-1}(2,3) \equiv 6H(4,1)+2 H(3,2)+6p H(4,2)+12 p H(5,1)\\
-4p\Big[H(5,1)+3H(4,2)+6H(4,1,1)+6H(3,2,1)+6H(3,1,2)+12H(3,1,1,1)\Big]\\
\equiv 6H(4,1)+2 H(3)H(2)-2H(2,3)\equiv 6H(4,1)-2H(2,3)
\end{multline*}
by using the reversal relation and Theorem~\ref{thm:wt6FMZV}.
This together with
Theorem~\ref{thm:reversal2} implies \eqref{equ:wt5superbity2z32}.
Further, \eqref{equ:wt5superbity2z32} shows that
\begin{align*}
\zeta^\star_{\calA_2}(2, 2, 1)
=&\,\zeta_{\calA_2}(2, 2, 1)+\zeta_{\calA_2}(2, 3)+\zeta_{\calA_2}(4, 1)+\zeta_{\calA_2}(5)\\
=&\,\zeta_{\calA_2}(2, 2, 1)+3\zeta_{\calA_2}(4, 1).
\end{align*}
Combining this with Theorem~\ref{thm:homogeneousWols2} and \eqref{equ:wt5superbity2z32} we have
\begin{align*}
&-6 \gb_{3}^2=\zeta_{\calA_2}(2)\zeta_{\calA_2}(2,1)=
2\zeta_{\calA_2}(2,2,1)+\zeta_{\calA_2}(4,1)+\zeta_{\calA_2}(2,3)+\zeta_{\calA_2}(2,1,2) \\
&=2\zeta^\star_{\calA_2}(2,2,1)-3\zeta_{\calA_2}(4,1)+\zeta_{\calA_2}(2,1,2)
=2\zeta^\star_{\calA_2}(2,2,1)-3\zeta_{\calA_2}(4,1)-3\gb_{3}^2
\end{align*}
by \eqref{equ:wt5superbity2z212}. This clearly implies \eqref{equ:wt5superbity2z122}.

Finally, the second ``='' of \eqref{equ:wt5superbity2z113} is an
easy consequence the concatenation relations
\eqref{equ:FMZV=concatFMZSV} since
$\zeta_{\calA_2}(3)=\zeta_{\calA_2}(1)=0$. Therefore
\begin{align*}
\zeta^\star_{\calA_2}(3,1,1)
=&\,\zeta_{\calA_2}(3,1,1)+\zeta_{\calA_2}(3,2)+\zeta_{\calA_2}(4,1)+\zeta_{\calA_2}(5)\\
=&\,- \zeta^\star_{\calA_2}(3,1,1)-\zeta_{\calA_2}(4,1)
\end{align*}
by \eqref{equ:wt5superbity2z32}. So the last  ``='' of \eqref{equ:wt5superbity2z113} is verified. This finishes the
proof of the theorem.
\end{proof}

Similarly we can obtain the following results in weight 6.
We leave the detailed computation to the interested reader
and provide essentially all the other values in Table~\ref{Table:wt6FMZVsuperbity2}.

\begin{thm} \label{thm:wt6FMZVleve2IDs}
In weight $6$ we have
{\allowdisplaybreaks
\begin{align*}
\zeta^\star_{\calA_2}(1,1,4)=\zeta^\star_{\calA_2}(4,1,1)
=&\zeta_{\calA_2}(1,1,4)=\zeta_{\calA_2}(4,1,1), \\ 
\zeta^\star_{\calA_2}(1,4,1)=\zeta_{\calA_2}(1,4,1)=&-2\zeta_{\calA_2}(4,1,1)+6 p\gb_{7}, \\  
\zeta^\star_{\calA_2}(2,3,1)=\zeta_{\calA_2}(1,3,2)=&3\zeta_{\calA_2}(4,1,1)-\frac{65}{4} p\gb_{7}, \\  
\zeta^\star_{\calA_2}(1,2,3)=\zeta_{\calA_2}(3,2,1)=&-2\zeta_{\calA_2}(4,1,1)+\frac{17}{4} p\gb_{7}, \\  
\zeta^\star_{\calA_2}(2,1,3)=\zeta_{\calA_2}(3,1,2)=&-\zeta_{\calA_2}(4,1,1)+11 p \gb_{7}, \\  
\zeta^\star_{\calA_2}(1,3,2)=\zeta_{\calA_2}(2,3,1)=&3\zeta_{\calA_2}(4,1,1)-\frac{9}{4} p\gb_{7}, \\  
\zeta^\star_{\calA_2}(3,2,1)=\zeta_{\calA_2}(1,2,3)=&-2\zeta_{\calA_2}(4,1,1)-\frac{11}{4} p\gb_{7}, \\  
\zeta^\star_{\calA_2}(3,1,2)=\zeta_{\calA_2}(2,1,3)=&-\zeta_{\calA_2}(4,1,1)+ 18  p \gb_{7}. \\  
\end{align*}}{}\noindent
\end{thm}

\begin{cor} \label{cor:wt6FMZVgenerators}
The FMZV space of weight $6$ superbity $2$ is given by
\begin{equation*}
\FMZ_{6,2}=\langle \zeta_{\calA_2}(4,1,1), p\gb_{7}\rangle.
\end{equation*}
\end{cor}
\begin{proof}
By Theorem~\ref{thm:homogeneousWols2} and Theorem~\ref{thm:FMZVdepth2superbity2}
all depth 1 and 2 values are $\Q$ multiples of  $p\gb_{7}$. All depth 3
values are presented in Theorem~\ref{thm:wt6FMZVleve2IDs}.
By duality \eqref{equ:dualityLevel2} and Theorem~\ref{thm:wt7FMZV}
for weight 7 values (which are all $\Q$ multiples of $\gb_{7}$) we see
that all values of larger depths are also linear
combinations of $\zeta_{\calA_2}(4,1,1)$ and $p\gb_{7}$. This completes the proof of the corollary.
\end{proof}

Using Maple, one can compute FMZVs of depth 2 up to weight 8
which are listed in the next theorem.
For essentially complete tables of values of weight up to 6 (inclusive) see
Tables~\ref{Table:wt2-4FMZVsuperbity2}, Tables~\ref{Table:wt5FMZVsuperbity2}
and Tables~\ref{Table:wt6FMZVsuperbity2}.

\begin{table}[!h]
{
\begin{center}
\begin{tabular}{  ||c|c||c|c||c|c|| } \hline
$\zeta_{\calA_2}(1,4)  $ & $3p\gbb_3^2-2z311$ &
$\zeta_{\calA_2}(4,1) $ & $-3p\gbb_3^2+2z311$ &
$\zeta_{\calA_2}(1,1,3)$ & $-z311$ \\ \hline
$\zeta_{\calA_2}(2,3)  $ & $-3p\gbb_3^2+4z311$&
$\zeta_{\calA_2}(1,2,2)$ & $-{\small \frac{3}{2}}p\gbb_3^2+3z311$ &
$\zeta_{\calA_2}(2,1,2)$ & $-3p\gbb_3^2$ \\ \hline
$\zeta_{\calA_2}(3,2)  $ & $3p\gbb_3^2-4z311$ &
$\zeta_{\calA_2}(2,2,1)$ & ${\small \frac{9}{2}}p\gbb_3^2-3z311 $&
$\zeta_{\calA_2}(1,3,1)$ & $0$ \\ \hline
\end{tabular}
\end{center}
}
\caption{FMZV of superbity 2 in weight 5, where
$z311=\zeta_{\calA_2}(3,1,1)$.}
\label{Table:wt5FMZVsuperbity2}
\end{table}

\begin{table}[!h]
{
\begin{center}
\begin{tabular}{  ||c|c||c|c||c|c||} \hline
$\zeta_{\calA_2}(6) $ & $ 6\gbb_7$ &
$\zeta_{\calA_2}(5,1) $ & $-10\gbb_7$ &
$\zeta_{\calA_2}(2,1,3) $ & $ 18\gbb_7-\ga$ \\ \hline
$\zeta_{\calA_2}(1,5) $ & $ 4\gbb_7$ &
$\zeta_{\calA_2}(1,2,3) $ & $ {\small -\frac{11}{4}}\gbb_7-2\ga$ &
$\zeta_{\calA_2}(2,2,2) $ & $ 2\gbb_7$ \\ \hline
$\zeta_{\calA_2}(2,4) $ & $ -10\gbb_7$ &
$\zeta_{\calA_2}(1,3,2) $ & $ {\small -\frac{65}{4}}\gbb_7+3\ga$ &
$\zeta_{\calA_2}(2,3,1) $ & $ {\small -\frac{9}{4}}\gbb_7+3\ga$ \\ \hline
$\zeta_{\calA_2}(3,3) $ & $ -3\gbb_7$ &
$\zeta_{\calA_2}(1,4,1) $ & $ 6\gbb_7-2\ga$ &
$\zeta_{\calA_2}(3,1,2) $ & $ 11\gbb_7-\ga$\\ \hline
$\zeta_{\calA_2}(4,2) $ & $ 4\gbb_7$  &
$\zeta_{\calA_2}(1,1,4) $ & $\ga$  &
$\zeta_{\calA_2}(3,2,1) $ & $ {\small \frac{17}{4}}\gbb_7-2\ga$ \\ \hline
\end{tabular}
\end{center}
}
\caption{FMZV of superbity 2 in weight 6, where
$\ga=\zeta_{\calA_2}(4,1,1)$.}
\label{Table:wt6FMZVsuperbity2}
\end{table}

\begin{thm}
Setting $\zeta_2=\zeta_{\calA_2}$ we have
{\allowdisplaybreaks
\begin{align*}
 \FMZ_{1,2}&=\langle 0\rangle,\quad
 \FMZ_{2,2}=\langle \zeta_2(2)\rangle,\quad
 \FMZ_{3,2}=\big\langle \zeta_2(1,2)\big\rangle,\\
 \FMZ_{4,2}&=\langle \zeta_2(2,2)\rangle,\quad
 \FMZ_{5,2}=\big\langle \zeta_2(2,3),\zeta_2(1,2,2)\big\rangle,\\
 \FMZ_{6,2}&=\big\langle \zeta_2(2,2,2),\zeta_2(1,2,3)\big\rangle,\\
 \FMZ_{7,2}&=\left\langle
{\zeta_2(\{1\}^5,2),\zeta_2(\{1\}^3,4),\zeta_2(1,6), \atop
        \bfzt_2(2,3,1,1), \bfzt_2(3,1,1,2), \bfzt_2(3,3,1)}\right\rangle,\\
 \FMZ_{8,2}&=\left\langle {\zeta_2(\{1\}^6,2),\zeta_2(\{1\}^4,4),\zeta_2(1,1,6),\zeta_2(1,2,5),
     \atop \bfzt_2(5,1,1,1)}\right\rangle,\\
 \FMZ_{9,2}&=\left\langle
 {
 {\zeta_2(\{1\}^7,2),\zeta_2(\{1\}^5,4),\zeta_2(\{1\}^3,6),
    \zeta_2(1,8), \zeta_2(1,2,6), }\atop
 \bfzt_2(1,3,2,3),\bfzt_2(2,1,3,3),
  \bfzt_2(2,3,1,3),\bfzt_2(1,2,5,1)
 }
\right\rangle.
\end{align*}}{}\noindent
\end{thm}

We expect that the bold-faced FMZVs are not really needed
for generating the corresponding spaces because of
the the following conjectured relations in $\calA_2$ which
we have verified numerically for the
first 1000 primes..
Setting $\zeta_2=\zeta_{\calA_2}$, then
{\allowdisplaybreaks
\begin{align*}
68\bfzt_2(2,3,1,1)=&\,-480\zeta_2(\{1\}^5,2)+716\zeta_2(\{1\}^3,4)-843\zeta_2(1,6)\\
34\bfzt_2(3,1,1,2)=&\,-585\zeta_2(\{1\}^5,2)+998\zeta_2(\{1\}^3,4)-1029\zeta_2(1,6)\\
68\bfzt_2(3,3,1)=&\,1730\zeta_2(\{1\}^5,2)-2480\zeta_2(\{1\}^3,4)+2319\zeta_2(1,6)\\
12\bfzt_2(5,1,1,1)=&\,-73\zeta_2(\{1\}^6,2)+12\zeta_2(1,1,6)-3\zeta_2(1,2,5)\\
84\bfzt_2(3,2,1,3)=&\,-995\zeta_2(1,8)+952\zeta_2(\{1\}^3,6)
-1288\zeta_2(\{1\}^5,4)\\
   &\, -437\zeta_2(\{1\}^7,2)-624\zeta_2(1,2,6)\\
924\bfzt_2(2,1,3,3)=&\,2509\zeta_2(1,8)-6356\zeta_2(\{1\}^3,6)
+5432\zeta_2(\{1\}^5,4)\\
    &\,-180\zeta_2(\{1\}^7,2)+801\zeta_2(1,2,6)\\
924\bfzt_2(2,3,1,3)=&\,-9424\zeta_2(1,8)-5824\zeta_2(\{1\}^3,6)
+11368\zeta_2(\{1\}^5,4)\\
    &\,-3807\zeta_2(\{1\}^7,2)+852\zeta_2(1,2,6)\\
36\bfzt_2(1,2,5,1)=&\,358\zeta_2(1,8)-248\zeta_2(\{1\}^3,6)
+464\zeta_2(\{1\}^5,4)\\
    &\,+5\zeta_2(\{1\}^7,2)+147\zeta_2(1,2,6).
\end{align*}}

In general, we have the following dimension conjecture in superbity 2
parallel to FMZVs of superbity 1.
See Table~\ref{Table:PadovanTable} for numerical support.
\begin{conj}\label{conj:FMZVsuperbity2}
Let $w$ be any positive integer. Set $d_{0,2}=1$ and $d_{w,2}=\dim \FMZ_{w,2}$.
Then $d_{1,2}=0,d_{2,2}=1$ and for all $w\ge 3$ we have
$$d_{w,2}=d_{w-2,2}+d_{w-3,2}.$$
\end{conj}

Now let us look at some numerical data given
in Table ~\ref{Table:PadovanTable}.
\begin{table}[!h]
{
\begin{center}
\begin{tabular}{|c|c|c|c|c|c|c|c|c|c|c|c|c|c| } \hline
$w$ &0 & 1 & 2& 3& 4& 5& 6& 7& 8& 9& 10& 11& 12\\ \hline
$\dim \MZV_w$ &1 & 0 &1 & 1 & 1 & 2& 2& 3 & 4 &5 &7 &9 &12\\ \hline
$\dim \FMZ_{w,1}$ &1 & 0 &0 & 1 & 0 & 1& 1& 1 & 2 &2 &3 &4 & 5 \\ \hline
$\dim \FMZ_{w,2}$ &1 & 0 &1 & 1 & 1 & 2& 2& 3 & 4 &5 &7 &9 &12\\ \hline
$\dim \FMZ_{w,3}$ &1 & 1 &1 &2 & 2 & 3& 4& 5 &7 &9 &12 &16 & 21 \\ \hline
$\dim \FMZ_{w,4}$ &2 & 1 &\textbf{1} & 3 & 3 & 5& 6& 8 & 11 &14 &19 &25  &  \\ \hline
$\dim \FMZ_{w,5}$ &2 & \textbf{1} &\textbf{2} & \textbf{3} & 5 & 7&  9& 12&  16&21 &  &  &  \\ \hline
\end{tabular}
\end{center}
}
\caption{Numerically verified conjectural dimensions of $\MZV_w$ and $\FMZ_{w,\ell}$ for $\ell\le 5$. Bold-faced numbers
are all one less than they're supposed to be by the conjectured recurrence formula.}
\label{Table:PadovanTable}
\end{table}

Note that at every superbity the conjectured recurrence relation
is always $f_{w}=f_{w-2}+f_{w-3}$ in general.
The only exceptions occur at superbities at least 4
but only for very small weights where we have ``too few'' values.
For example, at superbity 4 and weight 2 the dimension is supposed
to be 2, but because of the trivial relation
\begin{equation*}
\zeta_{\calA_4}(2)+2\zeta_{\calA_4}(1,1)= \zeta_{\calA_2}(1)^2=0\in \calA_4
\end{equation*}
the dimension is decreased to only 1. For another example,
at superbity $\ell=5$ or 6 we have another trivial relation
$$6\zeta_{\calA_\ell}(1,1,1)+\zeta_{\calA_\ell}(1,2)
+\zeta_{\calA_\ell}(2,1)+\zeta_{\calA_\ell}(3)=\zeta_{\calA_\ell}(1)^3=0.$$

It might be possible to modify the map $f_{\rm KZ}$ in Conjecture~\ref{conj:KanekoZagier}
to prove an isomorphism of at superbity $\ell|$.
\begin{conj}
For all positive integers $w$ and $\ell$, we have
$$\FMZ_{w,\ell} \cong \frac{\MZV_{w}}{\zeta(2)\MZV_{w-2}} \oplus \frac{\MZV_{w+1}}{\zeta(2)\MZV_{w-1}} \oplus \cdots  \oplus
\frac{\MZV_{w+\ell-1}}{\zeta(2)\MZV_{w+\ell-3}} .$$
\end{conj}
Clearly the image of $\zeta_{\calA_2}(\bfs)$
in the first component should be the symmetrized MZV $\zeta_\sha^\Sy(\bfs)$, and
the second components should involve some variation of $\zeta_\sha^\Sy(\bfs+\bfe_j)$
where $\bfs+\bfe_j$ means the $j$th component of $\bfs$ is increased by 1.
But we have to keep in mind the reversal
relation at superbity 2 is given by the following formula
\begin{equation}\label{equ:rev}
(-1)^{|\bfs|} \zeta_{\calA_2}(s_d,\dots,s_1)= \zeta_{\calA_2}(\bfs)+
p \sum_{j=1}^d s_j \zeta_{\calA_2}(\bfs+\bfe_j).
\end{equation}

\section{Finite Euler sums of superbity 1 and 2}\label{sec:dimConjES}
In \cite{TaurasoZh2010} we obtained a few results for FES. More similar
results can be found in \cite{HessamiPilehrood2Ta2013}.
We now turn to FES in arbitrary depth. As in Conjecture~\ref{conj:KanekoZagier}
we will link them to a symmetrized version of the Euler sums.
For any $s_1,\ldots, s_d\in\db$,
we may define the symmetrized Euler sums
\begin{align}
\zeta_\sha^\Sy(s_1, \ldots, s_d)=&\sum_{i=0}^d
\left(\prod_{j=1}^i (-1)^{s_j} \sgn (s_j)\right) \zeta_\sha(\bfp(s_i,\dots,s_1)) \zeta_\sha(\bfp(s_{i+1},\dots,s_d)), \label{equ:SESsha}\\
\zeta_\ast^\Sy(s_1, \ldots, s_d)=&\sum_{i=0}^d
\left(\prod_{j=1}^i (-1)^{s_j} \sgn (s_j)\right) \zeta_\ast(s_i,\dots,s_1) \zeta_\ast(s_{i+1},\dots,s_d), \label{equ:SESast}
\end{align}
where $\zeta_\sha$ and $\zeta_\ast$ are regularized
values given by Theorem~\ref{thm:2wayReg}. Similar to
Conjecture~\ref{conj:KanekoZagier} we propose

\begin{conj}\label{conj:FESdim}
Let $f_w=\dim_\Q \FES_{w,1}$ for $w\ge 1$. Then
\begin{equation*}
    \sum_{w=1}^\infty f_w t^w=\frac{t}{1-t-t^2}.
\end{equation*}
Moreover, $\ES_{w,1}$ has a basis
\begin{equation*}
    \Big\{\zeta_{\calA_1}(\bar1,a_1,\dots,a_r):
a_1,\dots,a_r\in\{1,2\}, a_1+\dots+a_r=w-1\Big\}.
\end{equation*}
\end{conj}

\begin{conj}\label{conj:KanekoZagierAltVersion}
There is an isomorphism
\begin{align*}
f_\FES:  \FES_{w,1}& \lra \frac{\ES_w}{\zeta(2)\ES_{w-2}} , \\
 \zeta_{\calA_1}(\bfs) & \lmaps \zeta_\ast^\Sy(\bfs).
\end{align*}
\end{conj}
We have checked this up to weight 5 under Conjecture~\ref{conj:A_1BernoulliAlgIndpt}
concerning the $\calA_1$-Bernoulli numbers and $\calA_1$-Fermat quotient $q_2$.

Since essentially the same proofs for Proposition~\ref{prop:Fait1.8}
given by Jarossay \cite{Jarossay2014}
work for Euler sums with the help of
level two Drinfeld associator, one may also  replace $\zeta_\ast$ by $\zeta_\sha$ in the conjecture. However, we don't know whether the space of Euler sums
$\ES_n$ can be generated by symmetrized Euler sums.

Conjecture~\ref{conj:KanekoZagierAltVersion} would imply that
$\dim_\Q \FES_{w,1}=F_w$ where $F_w$ are Fibonacci numbers:
$F_0=0, F_1=1, F_w=F_{w-1}+F_{w-2}$ for all $w\ge2$.
{}From what we have found so far we get the following table.
Let $\ES_w$ (resp.\ $\FES_{w,\ell}$) be the $\Q$-space
spanned by the Euler sums of weight $w$
(resp.\ finite Euler sums of weight $w$ and superbity $\ell$).

\begin{table}[!h]
{
\begin{center}
\begin{tabular}{|c|c|c|c|c|c|c|c|c|c| } \hline
$w$ &0 & 1 & 2& 3& 4& 5& 6& 7  \\ \hline
$\dim \ES_w$ &1 &1 &2 & 3 & 5 & 8& 13& 21    \\ \hline
$\dim \FES_{w,1}$ &1 &1 &1 &2 & 3 & 5 & 8& 13    \\ \hline
$\dim \FES_{w,2}$ &1 & 1 &2 & 4 & 7 & 12& 20& 33    \\ \hline
\end{tabular}
\end{center}
}
\caption{Conjectural dimensions of $\ES_w$ and $\FES_{w,\ell}$ for $\ell\le 2$.}
\label{Table:FibonacciTable}
\end{table}
In Table~\ref{Table:FibonacciTable}, we provide $\dim \ES_w$
for comparison purpose. We also verified
$\dim \FES_{w,1}$ and $\dim \FES_{w,2}$ in the table
numerically for the first 1000 primes.

As further support of Conjecture~\ref{conj:FESdim}, we have
\begin{thm}
For all $w\le 7$, We have
\begin{equation*}
 \FES_w =\big\langle \zeta_{\calA_1}(\bar1,a_1,\dots,a_r):
    a_1,\dots,a_r\in\{1,2\}, a_1+\cdots+a_r=w-1 \big\rangle.
\end{equation*}
\end{thm}
\begin{proof}
This is proved with Maple. We found that the double shuffle
relations defined in Definition~\ref{defn:FESdbsf} are
insufficient to produce all the relations. However, with
reversal relations we can find the generating set as given
in the theorem, for all $w\le 7$.
\end{proof}

In the classical setting, we know the double shuffle relations
do not generate all the linear relations among Euler sums. For example, in
proving the identities $\zeta(\{3\}^n)=8^n\zeta(\{\ol2,1\}^n)$ for all $n\ge 1$,
we need to use the distribution relations (see \cite[Remark 3.5]{Zhao2010a}).

\begin{eg}
For FESs, the first missing reversal relation (i.e., not provable by the double shuffle relations),
which appears in weight two already, is given by
$$\zeta_{\calA_1}(1,\bar1)+\zeta_{\calA_1}(\bar1,1)=0.$$
In weight 3, we need two reversal relations
$$\zeta_{\calA_1}(\bar2, 1)=\zeta_{\calA_1}(1, \bar2),  \quad{and}\quad
 \zeta_{\calA_1}(1,1,\bar1)=\zeta_{\calA_1}(\bar1,1,1).$$
\end{eg}

\end{document}